\documentclass[12pt,a4paper]{article}

\usepackage{amssymb,amsmath,amsthm,xspace,epsfig, amsfonts,amscd, bm}

\usepackage[left=1.5cm, right=1.5cm, top=1.5cm]{geometry} 

\usepackage{mathrsfs}
\usepackage{epsfig}
\usepackage{tikz-cd}

\begin{document}

\newcommand{\N}{\mathbb{N}}
\newcommand{\R}{\mathbb{R}}
\newcommand{\Z}{\mathbb{Z}}
\newcommand{\Q}{\mathbb{Q}}
\newcommand{\C}{\mathbb{C}}
\newcommand{\PP}{\mathbb{P}}
\newcommand{\SSS}{\mathbb{S}}

\newcommand{\LL}{\Bbb L}
\newcommand{\OO}{\mathcal{O}}

\newcommand{\esp}{\vskip .3cm \noindent}
\mathchardef\flat="115B

\newcommand{\lev}{\text{\rm Lev}}

\newcommand{\essinf}{\text{\rm essinf\,}}
\newcommand{\jolt}{\text{\rm jolt}}

\def\ut#1{$\underline{\text{#1}}$}
\def\CC#1{${\cal C}^{#1}$}
\def\h#1{\hat #1}
\def\t#1{\tilde #1}
\def\wt#1{\widetilde{#1}}
\def\wh#1{\widehat{#1}}
\def\wb#1{\overline{#1}}

\def\restrict#1{\bigr|_{#1}}

\def\hu#1#2{\mathsf{U}_{fin}\bigl({#1},{#2}\bigr)}
\def\ch#1#2{\left(\begin{array}{c}#1 \\ #2 \end{array}\right)}

\newtheorem{lemma}{Lemma}[section]

\newtheorem{thm}[lemma]{Theorem}
\newtheorem*{thm*}{Theorem}
\newtheorem*{lemma?}{Lemma ??}

\newtheorem{defi}[lemma]{Definition}
\newtheorem{conj}[lemma]{Conjecture}
\newtheorem{cor}[lemma]{Corollary}
\newtheorem{prop}[lemma]{Proposition}
 
\newtheorem{prob}[lemma]{Problem}
\newtheorem{qu}[lemma]{Question}
\newtheorem{q}[lemma]{Questions}
\newtheorem*{rem}{Remark}
\newtheorem{examples}[lemma]{Examples}
\newtheorem{example}[lemma]{Example}

\theoremstyle{remark}
\newtheorem{claim}{Claim}[lemma]
\renewcommand{\theclaim}{\arabic{section}.\arabic{lemma}.\arabic{claim}}
\renewcommand{\theequation}{\arabic{section}.\arabic{equation}}

\title{Filling the gaps in an unpublished example of Nyikos: a 
       countably compact non-compact manifold under $\clubsuit_C$}
%\date{}
\author{Mathieu Baillif}
   
\maketitle

\begin{abstract}
   \footnotesize
   We provide details for a Theorem which is only available on a unfinished preliminary draft of P. Nyikos:
   the existence 
   under $\clubsuit_C$
   of a hereditarily collectionwise normal countably compact non-compact manifold 
   which does not contain a copy of $\omega_1$.
   This shows in particular that {\bf MA + $\neg$CH} does not imply 
   the existence of a copy of $\omega_1$ in a
   countably compact non-compact manifold (it is known that {\bf PFA} does imply it).
   The said manifold is obtained from a principal $\SSS^1$-bundle over the long ray, and
   some structural Theorems for these spaces (due to Nyikos as well) are also proved.
   We show that the same type of theorems hold for
   $n$-to-$1$ closed preimages of $\omega_1$ and $\Z_n$-``bundles'' over $\omega_1$, where
   $\Z_n$ is the additive group of integers modulo $n$.
   A small generalization of $\clubsuit_C$ is also quickly investigated.
\end{abstract}

\tableofcontents

%%%%%%%%%%%%%%%%%%%%%%%%%%%%%%%%%%%%%%%%%%%%%%%%%%%%%%%%%%%%%%%%%%%%%%%%%%%%%%%%%%%%%%%%%%%%%%%%%%%%%%%%%%%
%%%%%%%%%%%%%%%%%%%%%%%%%%%%%%%%%%%%%%%%%%%%%%%%%%%%%%%%%%%%%%%%%%%%%%%%%%%%%%%%%%%%%%%%%%%%%%%%%%%%%%%%%%%
%%%%%%%%%%%%%%%%%%%%%%%%%%%%%%%%%%%%%%%%%%%%%%%%%%%%%%%%%%%%%%%%%%%%%%%%%%%%%%%%%%%%%%%%%%%%%%%%%%%%%%%%%%%5
\section{Introduction}

In this note, a {\em manifold} is a connected Hausdorff space locally homeomorphic to $\R^n$ for some $n$
(which is fixed by connectedness). We use standard terminology, 
see the end of this introduction if in need of definitions.
The two following theorems were proved quite some time ago.

\begin{thm}[{\bf Countable compactness in manifolds}]
   \ \\
   (a) (Z. Balogh, 1989 \cite[Corollary 2.7]{Balogh:1989}) 
       Under {\bf PFA}, any countably compact non-compact manifold contains a copy of $\omega_1$.\\
   (b) (M.E. Rudin, 1976 \cite[Corollary 2]{RudinZenor}) Under $\diamondsuit$, there is a 
       perfectly normal non-compact countably 
       compact manifold which does not contain any
       copy of $\omega_1$.
   \label{thm:ctblecomp}
\end{thm}

Countably compact metrizable manifolds are compact, hence the manifolds above are non-metrizable.
Here, {\bf PFA} is the proper forcing axiom and $\diamondsuit$ is Jensen's diamond axiom.
Since then, many variants of Rudin's construction have been given, but to our knowledge
all published ones use at least some flavor of $\diamondsuit$ incompatible with {\bf MA + $\neg$CH}, that is,
Martin's axiom plus the negation of the continuum hypothesis. Since
{\bf PFA} implies {\bf MA + $\neg$CH},
the following question is natural:

\begin{q}
   Is {\bf MA + $\neg$CH} enough for (a)~?
\end{q}

P. Nyikos answered this question in the negative.
\begin{thm}[{Nyikos, circa 2010 \cite{Nyikos:Antidiamond}}]\ \\
   \label{thm:intro}
   There is a countably compact non-compact manifold without any
   copy of $\omega_1$ under $\clubsuit_C$.
   The manifold is moreover hereditarily collectionwise normal.
\end{thm}
The axiom $\clubsuit_C$ is not killed by any ccc forcing, and thus may coexist with {\bf MA + $\neg$CH}.
This construction also sheds some light on other results concerning 
strengthenings of normality in manifolds, such as those ones (see the references for definitions and details).

\begin{thm}[{\bf Hereditary normality in manifolds}]
   \ \\
   (a) (M.E. Rudin , 1979 \cite{Rudin:1979}) 
       Under {\bf MA + $\neg$CH}, every perfectly normal manifold is metrizable.\\
   (b) (A. Dow and F.D. Tall, 2016 \cite{DowTall:2016b}) 
       Under {\bf PFA(S)[S]}, any hereditarily normal manifold of dimension $>1$
       is metrizable.\\
   (c) (Nyikos, 2004 \cite{Nyikos:correction}) Under {\bf PFA}, every normal 
       hereditarily collectionwise Hausdorff manifold is metrizable.\\
   (d) (M.E. Rudin and P. Zenor, 1976 \cite[Theorem 1]{RudinZenor}) 
       There is a perfectly normal non-metrizable manifold without any
       copy of $\omega_1$ under {\bf CH}.
\end{thm}

Nyikos' Theorem \ref{thm:intro} appears in Sections 4 and 5 of an unfinished draft \cite{Nyikos:Antidiamond}
which he made available on his webpage.
There is a sketch of a proof through a series of lemmas (which are given without proof), while
the construction itself of the manifold is described in some more detail.
Since P. Nyikos unfortunately passed away before completing his draft, we thought
that it might be a good idea to try to fill the gaps and this note is the result of our efforts.
Most of our presentation follows Nyikos' roadmap, 
we only add some lemmas of our own when we were not able to 
cut directly through the arguments, and adapt a bit his notations or conventions to our taste. 
\\
There are two notable exceptions to our loyalty to Nyikos' plan. The first one is Section \ref{sec:tool},
where we isolate a kind of combinatorial result on ladder systems which makes Nyikos' construction more easily 
translatable in other settings. (Another reason is that we struggled to find the correct
definition for making the crucial Lemma \ref{lemma:jolt} work, as none is given in Nyikos' notes.
We thus felt the need to temporarily focus on the combinatorial aspects of our problem, away from
the more geometrical ones.)
\\
The other exception is Section \ref{sec:finitetoone}. What happened is that
during the course of writing these notes, we thought that it would be nice to say something about
a related class of spaces that Nyikos also investigated, in the same or in other papers and drafts 
(e.g. in \cite{Nyikoscoherent}):
the finite-to-$1$ closed preimages of $\omega_1$.
According to some remarks on his drafts,
Nyikos had planned to present his manifold results together with similar ones about these spaces,
and we include some of them. We also allowed ourselves 
a bit of leeway to explore some things on our own and
as is often the case, the small set of remarks we wanted to add grew up taking almost a third of the note
and might contain new results, 
at least ones that we could not spot in the literature or in Nyikos' drafts.
Some of them are immediate consequences of the combinatorial results of Section \ref{sec:tool}.
These possible novelties have the advantage of giving us the sensation of a little original contribution
and a small boost to our self-esteem, which
gave us enough momentum to add a short last section (Section \ref{sec:variationclub})
where we investigate a variation of $\clubsuit_C$ (which might also be new, although a first proof that it is consistent
was provided by G. Goldberg and A. Lietz, see below).
Due to a misunderstanding of a definition in Nyikos' draft, we thought for some time 
that this variation was needed in the proof of Theorem \ref{thm:intro}.
To summarize: the only parts of this note that are not based on Nyikos drafts are Subsection \ref{subsec:notprincipal}
and Section \ref{sec:variationclub}.

\vskip .3cm
This note is organized as follows. 
Nyikos' manifold is actually a principal $\SSS^1$-bundle over the long ray $\LL_+$;
Section \ref{sec:generalities} contains generalities about these spaces and 
a structure theorem for them. This enables us to define the order of such a bundle,
which is an integer or $\infty$, and to prove that bundles of order $>1$ do not possess a copy of $\omega_1$.
Our goal then becomes showing that there are bundles of any order under $\clubsuit_C$. 
Section \ref{sec:tool} contains the aforementioned combinatorial results on 
(arbitrary) ladder systems,
and Section \ref{sec:constr} uses these results to 
construct ``arcs'' $a_\alpha$ on the set $\LL_+\times\SSS^1$ which engender 
(in a precise sense given in this section) the topology of a principal $\SSS^1$-bundle on this set.
This section is probably filled with way too many details for a lot of readers' taste,
but we were not confident enough to sprint to conclusions more quickly.
Then, Nyikos's theorem (or rather, the fact that there are bundles
of any order under $\clubsuit_C$) is proved in Section \ref{sec:usingclub}.
As said above, Section \ref{sec:finitetoone} 
contains some musing about finite-to-$1$ closed preimages of $\omega_1$.
In particular, we investigate ``discrete'' equivalents of principal $\SSS^1$-bundle over the long ray $\LL_+$:
principal $\Z_n$-bundles over $\omega_1$ (where $\Z_n$ is the group of integers modulo $n$).
Of course, there is something a bit silly in the idea of looking at bundles above a (highly) disconnected space,
but $\omega_1$ has a kind of ``rigidity'' that makes some things possible.
The contents of Section \ref{sec:variationclub} were already discussed above: a generalization of $\clubsuit_C$.

\vskip .3cm
Our terminology for topological terms is standard and follows e.g. \cite{Engelking}.
We denote ordered pairs by $\langle\cdot,\cdot\rangle$, reserving 
parenthesis $(\cdot,\cdot)$ for open intervals in totally ordered spaces.
We use the symbol $\upharpoonright$ for restrictions (of a function or a topology to a subspace, etc). 
We recall that the word {\em club} is a shorthand
for ``closed and unbounded''. The greak letters $\alpha,\beta,\gamma,\eta,\xi$ are used 
exclusively for ordinals. 
Any ordinal is the set of its predecessors, hence $\beta<\alpha$ and $\beta\in\alpha$
mean the same thing. Similarly, as sets, $\alpha$ is the same as the interval $[0,\alpha)$ taken in $\omega_1$.
We generally favor the notations $\beta<\alpha$ and $[0,\alpha)$ to avoid unnecessary confusions
between variables and domains of some functions.
The {\em long ray} $\LL_+$ is the space $\omega_1\times[0,1)$ with
lexicographic order topology and the smallest point removed.
It is a $1$-dimensional manifold.
It will be convenient to write the members of $\LL_+$ as $\alpha + r$, for $r\in[0,1)$ and $\alpha\in\omega_1$,
and when needed we see $\omega_1$ as a subset of $\LL_+$ (all the points with $r=0$).
\\
We let $\Lambda$ be the club subset of limit ordinals in $\omega_1$, and $\Lambda_2$ be that of limits of limits.
If $E\subset\omega_1$, we write $E'$ for the subset of its limit points.
A {\em ladder system} on $\omega_1$ is a sequence $\langle L_\alpha\,:\,\alpha\in\Lambda\rangle$
where each $L_\alpha$ is a strictly increasing sequence of ordinals with limit $\alpha$.
We recall that $\clubsuit_C$ is the following axiom:

\begin{center}
    \begin{minipage}{.05\textwidth} 
       $\clubsuit_C$ : 
    \end{minipage}
    \begin{minipage}{.85\textwidth}
                     There is a ladder system $\langle L_\alpha\,:\,\alpha\in\Lambda\rangle$ such that
                     for each club $C\subset\omega_1$, the set
                     $\{\alpha\in\Lambda\,:\,L_\alpha \subset C\}$ is stationary.
    \end{minipage} 
\end{center}
Such a ladder system is called a $\clubsuit_C$-sequence. 
The following classical lemma will be used many times in the proofs.

\begin{lemma}[{\bf Fodor's lemma}]
   If $S\subset\omega_1$ is stationary and $f:S\to\omega_1$ satisfies $f(\alpha)<\alpha$,
   then there is a sationary $S_0\subset S$ and $\beta\in\omega_1$ such that $f(\alpha) = \beta$
   for each $\alpha\in S_0$.
\end{lemma}

We also use several times (without acknowledging it)
the facts that a finite or countable union of non-stationary sets is non-stationary, and 
that a countable intersection of club subsets of $\omega_1$ or $\LL_+$ is club.

%%%%%%%%%%%%%%%%%%%%%%%%%%%%%%%%%%%%%%%%%%%%%%%%%%%%%%%%%%%%%%%%%%%%%%%%%%%%%%%%%%%%%%%%%%%%%%
%%%%%%%%%%%%%%%%%%%%%%%%%%%%%%%%%%%%%%%%%%%%%%%%%%%%%%%%%%%%%%%%%%%%%%%%%%%%%%%%%%%%%%%%%%%%%%
%%%%%%%%%%%%%%%%%%%%%%%%%%%%%%%%%%%%%%%%%%%%%%%%%%%%%%%%%%%%%%%%%%%%%%%%%%%%%%%%%%%%%%%%%%%%%%
\section{Generalities on principal $\SSS^1$-bundles over $\LL_+$}
\label{sec:generalities}

Nyikos' example is built as a principal $\SSS^1$-bundle over $\LL_+$, with $\SSS^1$ the circle.
A quick reminder on principal $\mathbb{S}^1$-bundles $E$ over a space $X$:
\begin{itemize}
   \item[$\bullet$] It is a fiber bundle, i.e. there is a continuous bundle map
                    $\pi:E\to X$ such that for each $x\in X$ there is an open $U\ni x$
                    and a fiber preserving homeomorphism 
                    $f_U:\pi^{-1}(U)\to U\times\SSS^1$, i.e.
                    $f_U$ sends $\{y\}\times \SSS^1$ to $\pi^{-1}(\{y\})$ for each $y\in U$. 
                    (This is sometimes summarized by the phrase ``$\pi:E\to X$ is locally trivial''.)
   \item[$\bullet$] There is a free transitive continuous group action by $\SSS^1$ on $E$ which leaves
                    fibers $\pi^{-1}(\{x\})$ invariant.          
\end{itemize}
We sometimes abuse the terminology and 
say that a fiber (or principal) bundle $\pi:M\to\LL_+$
has a topological property $\mathcal{P}$ whenever the total space $M$ has $\mathcal{P}$,
or say ``the (fiber or principal) bundle'' instead of ``the total space of the (fiber or principal) bundle''.
We see $\SSS^1$ as the reals modulo $2\pi$, with the group action given by the addition modulo $2\pi$.
Recall that any fiber bundle or principal $\SSS^1$-bundle on $\R$ is trivial, that is, there is a 
fiber preserving homeomorphism onto $\R\times\SSS^1$.

\vskip .3cm
\noindent
{\bf Notation:}\ \\
  If $\pi:M\to\LL_+$ is an $\SSS^1$-principal bundle and $p\in\LL_+$, we write $M_p = \pi^{-1}(\{p\})$. 
  If $x\in M$ and $\theta\in[0,2\pi)$, we denote by $x+\theta$ the point in $M_{\pi(x)}$ that is $\theta$ radians
  from $x$ in the positive direction.
\vskip .3cm
\noindent
Notice that $x+\theta$ is defined by the action of $\SSS^1$ on $M$, and 
since the group action is free, $x+\theta$ is uniquely defined.
Principal $\SSS^1$-bundles have the property that convergent sequences preserve angular separation:
\begin{lemma} 
   \label{lemma:x_ntheta_n}
   Let $\pi:M\to\LL_+$ be a principal $\SSS^1$-bundle.
   If $x_n\to x\in M$ and $y_n=x_n+\theta_n$, then $\theta_n\to\theta$ iff $y_n\to y$. 
\end{lemma}
\begin{proof}
   The group action is continuous, hence $x_n + \theta$ converges to $x+\theta$.
\end{proof}
We may thus treat each fiber $M_p$ as an isometric copy of $\SSS^1$ 
with distance given by angular separation.
It follows that the trivializing local fiber-preserving maps 
can be chosen such that their restrictions to the various rings
$\{r\}\times\SSS^1$ are isomorphisms to the images $M_p$.
The following ``almost classical''
theorem shows that there is a ``canonical way'' to present the data of an $\SSS^1$-bundle over $\LL_+$.
We recall that angle additions are modulo $2\pi$.

\begin{thm}
   \label{thm:canonical}
   Let $\pi:M\to\LL_+$ be a principal $\SSS^1$-bundle.
   Then up to bundle-isomorphism we may assume that
   the total space $M$ has underlying set $\LL_+\times\SSS^1$,
   $\pi$ is the projection on the first factor,
   the action of $\SSS^1$ is given by $\theta\cdot\langle p,\phi\rangle = \langle p,\phi+\theta\rangle$ and there are 
   trivializing maps $f_\alpha:(0,\alpha+1)\times\SSS^1\to \pi^{-1}((0,\alpha+1))$ (intervals are taken
   in $\LL_+$) such that $\pi\circ f_\alpha(\langle p,\phi\rangle) = p$.  
\end{thm}
The proof uses well known techniques, so we just sketch it.
\begin{proof}[Sketch of the proof]
   First, each fiber $M_p$ is isomorphic to $\SSS^1$, so we may assume that the total space
   has underlying set $\LL_+\times\SSS^1$.
   For each $p\in\LL_+$ there are some $q_0(p)<p<q_1(p)$ and a trivializing map
   $f_p:(q_0(p),q_1(p))\times\SSS^1\to\pi^{-1}((q_0(p),q_1(p)))$. By Fodor, 
   there is a stationary set $S\subset\omega_1\subset\LL_+$
   and some $\beta\in\omega_1$ such that $q_0(\alpha) = \beta$ for each $\alpha\in S$.
   Recall that any interval segment $(a,b)\subset\LL_+$ is homeomorphic to $\R$. 
   Using the fact that we may ``patch up'' trivial bundles together,
   we may assume that $\beta = 0$. For each $p\in\LL_+$, set $f_p = f_\alpha$ for the smallest $\alpha\in S$ bigger
   than $p$.
   This shows that we may assume that we have trivializing fiber-preserving maps
   $f_\alpha:(0,\alpha+1)\times\SSS^1\to\pi^{-1}((0,\alpha+1))$.
   The map $j_\alpha(q) = \pi\circ f_\alpha(\langle q,0\rangle)$ is a self-homeomorphism
   of $(0,\alpha+1)\subset\LL_+$, 
   letting $g_\alpha(\langle q , \theta\rangle) = f_\alpha(\langle j_\alpha^{-1}(q),0\rangle)$,
   then $g_\alpha$ is a fiber-preserving homeomorphism satisfying
   $\pi\circ g_\alpha(\langle q,\phi\rangle) = q$.
\end{proof}

\begin{lemma}
   \label{lemma:ctblycpct}
   Any $\SSS^1$-fiber bundle $\pi:M\to\LL_+$
   is ``almost'' countably compact, in the sense that
   $\pi^{-1}([p,\omega_1))$ is countably compact for each $p\in\LL_+$.
\end{lemma}
\begin{proof}
   Any countable subset of $\pi^{-1}([p,\omega_1))$ is contained in $\pi^{-1}([p,\alpha])$ for some $\alpha$,
   which is homeomorphic to $[p,\alpha]\times\SSS^1$ and hence compact.
\end{proof}

We say that $F\subset M$ is unbounded iff $\pi(M)$ is unbounded in $\LL_+$.
The following structural theorem showcases the intrinsic symmetry of principal $\SSS^1$-bundles.

\begin{thm}[{\cite[Theorem 4.1]{Nyikos:Antidiamond}}]
   \label{thm:Nyikos}
   Let $M$ be a principal $\SSS^1$-bundle over $\LL_+$, with bundle map $\pi:M\to\LL_+$ and homeomorphisms
   $f_\alpha:(0,\alpha+1)\times\SSS^1\to \pi^{1}((0,\alpha+1))$ as above. 
   Let $F\subset M$ be club. Then there is a club $C\subset\omega_1$ such that
   exactly one of the two alternatives below hold.\\
   (a) $F\supset\pi^{-1}(C)$;\\
   (b) There is $n\in\omega$ and a club $E\subset F$ such that $|E\cap M_\alpha|=n$ for each $\alpha\in C$,
       and if $x\in M_\alpha\cap E$, then $M_\alpha\cap E = \{x+\frac{2\pi\cdot k}{n}\,:\,k=0,\dots,n-1\}$.
\end{thm}

(Again, additions are modulo $2\pi$.)
We prove this theorem following Nyikos' series of lemmas
given in \cite{Nyikos:Antidiamond}, although we could not always find
quick arguments 
and had to add some (rather inelegant) lemmas of our own.
The final steps of the proof of (b) seem particularly clumsy and contorted,
but we have not found anything better.
First, an almost triviality.
\begin{lemma}
   \label{lemma:club}
   If $F\subset M$ is club, then $\pi(M)$ is club and $F\cap M_p$ is compact for each $p\in\LL_+$.
\end{lemma}
\begin{proof}
   Immediate by countable compactness (in the sense of Lemma \ref{lemma:ctblycpct}).
\end{proof}

\begin{defi}
   Let $Y\subset M$ and $p\in\LL_+$. A closed gap in $Y\cap M_p$ is a pair of points $x,x+\theta\in M_p\cap Y$
   ($0<\theta\le 2\pi$) such that $x+\phi\not\in Y$ for each $0<\phi<\theta$. We call $\theta$ the width of the gap. 
   An open gap of width $\theta$ in $Y\cap M_p$ is a set $\{x+\delta\,:\,0<\delta<\theta\}\subset Y\cap M_p$ for some 
   $x$. We abbreviate ``gap of width $\theta$'' by ``$\theta$-gap''.
   For each $\alpha\in\pi(Y)\cap\omega_1$, we set 
   $$ h_Y(\alpha) = \max\{\theta\,:\,\text{there is a closed $\theta$-gap in }F\cap M_\alpha\}.$$
\end{defi}
Notice that if $\theta = 2\pi$, then $x=x+\theta$ and $F\cap M_p$ is a singleton.
Also, $Y\cap M_p$ has only finitely many closed 
gaps of width $>\theta>0$, hence the max is attained in the definition of $h_Y$.
%If $\phi_0,\phi_1\in\SSS^1$, the interval $(\phi_0,\phi_1)$ is understood as the set
%of all points reached when going from $\phi_0$ to $\phi_1$ in the positive direction. (Hence, we may
%have $\phi_0>\phi_1$ seen as members of $[0,2\pi)$.)
\begin{defi}
  \label{defi:strips}
  Given $p<q<\alpha+1$ in $\LL_+$ and $\phi,\theta\in\SSS^1$ we let $S_\alpha(p,q,\phi,\theta)$ be 
  the image under $f_\alpha$ of the strip $(p,q)\times(\phi,\phi+\theta)$. 
\end{defi}
We also call these sets ``strips'',
  somewhat informally.

\begin{lemma}
   \label{lemma:opengap}
   Let $U\subset M$ be open such that there is $\theta > 0$ and a stationary set $S$ of $\alpha$
   such that $U\cap M_\alpha$ contains $n$ disjoint open $\theta$-gaps.
   Then, there is $\beta\in\omega_1$ such that for each $\alpha>\beta$
   there are disjoint open $U^\alpha_i\subset U\cap \pi^{-1}((0,\alpha+1))$, $i=1,\dots,n$, 
   $U_i = S_\alpha(\beta,\alpha,\phi_i,\theta)$,
   whose intersection with $M_\gamma$ is an 
   open $\theta$-gap for each $\beta<\gamma\le\alpha$. 
\end{lemma}

\begin{proof}
   Fix some strictly positive $\theta'<\theta$.
   Notice that in $(0,\alpha+1)\times\SSS^1$ with the product topology, if an open  
   set $V$ intersected with $\{\alpha\}\times\SSS^1$ contains $n$ $\theta$-gaps
   and $\alpha\in\Lambda$, then by taking compact intervals of length $\theta'$ strictly contained
   in these gaps and the definition of the product topology,
   there is $\beta<\alpha$ such that $V$ contains $n$ 
   disjoint
   strips $(\beta,\alpha]\times(\phi_i,\phi_i+\theta')$.   
   Since $f_\alpha$ is an isometry on the fibers $M_\alpha$, 
   there is $\beta(\theta',\alpha)<\alpha$ and disjoint
   $U_{\alpha,i}(\theta') = S_\alpha(\beta(\theta',\alpha),\alpha,\phi_i,\theta')$
   such that $U$ contains all these strips.
   By Fodor, there is a stationary $S_0\subset S$
   and $\beta(\theta')\in\omega_1$ such that $\beta(\theta',\alpha)=\beta(\theta')$ for each $\alpha\in S'$.
   If $\eta>\beta$,
   set $U_{\eta,i}(\theta')=U_{\alpha,i}(\theta')$ for the minimal $\alpha\ge\eta$ in $S_0$, we thus showed:
   \begin{equation}
      \label{eq:theta'}
         \alpha>\beta(\theta') \quad\Longrightarrow\quad U_{\alpha,i}(\theta')\subset U\,\forall i=1,\dots,n.
   \end{equation}
   Set $\beta = \sup_{m\in\omega}\beta(\theta-\frac{1}{m})$, then if $\alpha>\beta$,
   by \ref{eq:theta'} $U\cap\pi^{-1}((\beta,\alpha))$ 
   contains $n$ disjoint strips of widths $\theta-1/m$ for each $m$.
   Since the fibers have finite length $2\pi$, 
   except for finitely many of them, strips of bigger widths contain those of smaller widths.
   It follows that $U\cap\pi^{-1}((\beta,\alpha))$ contains at least $n$ disjoint strips
   of width $\theta$.
\end{proof}

\begin{lemma}
   \label{lemma:h_F=0}
   Let $F\subset M$ be closed. If $h_F(\alpha) = 0$ and $M_\alpha\cap F\not=\varnothing$, then 
   $M_\alpha\subset F$.
\end{lemma}
\begin{proof}
   Since $F\cap M_\alpha$ has no gap of strictly positive width, $F$ is dense in $M_\alpha$.
   By closedness, $F\supset M_\alpha$.
\end{proof}

If $g:E\to\R$ is a %(not necessarily continuous) 
function, where $E\subset\LL_+$, we define its essential infimum as
$$ \text{essinf\,}g =\sup_{\alpha\in\omega_1}\{\inf g\upharpoonright (E\cap[\alpha,\omega_1))\}.$$
\begin{lemma}
   \label{lemma:essinf}
   If $F\subset M$ is club, then $h_F(\alpha)$ attains it essential infimum $\theta_0$ on a club set $C$.
\end{lemma}
\begin{proof}
   We may assume that $D=\pi(F)\subset\omega_1\subset\LL_+$.
   Assume first that $\theta_0 = 0$.
   For each $\alpha_0$ we may choose $\alpha_n\in F$, $n>0$, such that $h_F(\alpha_n) < \frac{1}{n}$
   and $\alpha_{n+1}\ge\alpha_n$. Let $\alpha$ be the supremum of the $\alpha_n$.
   Then we must have $h_F(\alpha)=0$ and $F\supset M_\alpha$ by Lemma \ref{lemma:h_F=0}.
   There is thus a club set of $\alpha$ such that $F\supset M_\alpha$ and the lemma follows.\\
   Assume now that $\theta_0>0$. 
   By the definition of the essential infimum, for each $n$, there is $\beta_n$ such that 
   for each $\beta\in D$, $\beta>\beta_n$, we have $h_F(\beta)\ge\theta_0 -\frac{1}{n}$.
   Let $\beta=\sup_n\beta_n$, by restricting to $D-[0,\beta]$ we may assume 
   that $h_F(\alpha)\ge\theta_0$ for each $\alpha\in D$.
   \\
   As above, for each $\alpha_0$ 
   we may choose for each $n$ an $\alpha_n$ such that $\theta_0 \le h_F(\alpha_n) < \theta_0 + \frac{1}{n}$
   and $\alpha_{n+1}\ge\alpha_n$. Let $\alpha$ be the supremum of the $\alpha_n$.
   For each point in $F\cap M_{\alpha_n}$ there is another point at distance between $\theta_0$ and
   $\theta_0+\frac{1}{n}$, 
   by Lemma \ref{lemma:x_ntheta_n} each point of $F\cap\alpha$ has another one at distance exactly $\theta_0$.
   Since $h_F(\alpha)\ge\theta_0$, there is at least one closed gap of width $\theta_0$ in $M_\alpha$ (and none
   of bigger width).
\end{proof}

Of course, there might be gaps of any kind of widths on the same club $F$, for instance in the
trivial bundle $\LL_+\times\SSS^1$ (product topology). 

\begin{defi}
   Let $F\subset M$ be club such that $0<\essinf h_F = \theta_0 = h_F(\alpha)$ 
   for $\alpha\in C = \pi(F)$. We define $g_F:C\to\omega$
   as the number of gaps of width $\theta_0$ in $F\cap M_\alpha$.
\end{defi}

\begin{lemma}
   \label{lemma:ngaps}
   If $F\subset M$ is club with $\essinf h_F = \theta > 0$, 
   then there is $n\in\omega$ and a club 
   $C\subset\omega_1$ such that $g_F$ is constant with value $n$ on $C$.
   Moreover, the subset $F_0$ of points in $\pi^{-1}(C)\cap F$ that are boundary points
   of the $n$ $\theta$-gaps is club.
\end{lemma}
\begin{proof}
   Let $C$ be club such that $h_F(\alpha) = \theta$ on $C$, given by the previous lemma.
   Let $n$ be maximal such that $S=\{\alpha\in C\,:\,g_F(\alpha) = n\}$ is stationary.
   Then $\{\alpha\in C\,:\,g_F(\alpha) \not= m\}$ contains a club when $m>n$, 
   hence there is a club $D\subset C$ such that $h_F(\alpha)\le n$ on $D$.
   Let $U = M-F$. By Lemma \ref{lemma:opengap}, there is 
   $\beta$ such that $U\cap M_\gamma$ contains at least $n$ disjoint open gaps of width $\theta$
   when $\gamma>\beta$.
   Since $h_F(\alpha)=\theta$ on $D$, 
   there must be points of $F\cap M_\alpha$ at the extremities of each of these open gaps, hence 
   $F$ contains at least $n$, and thus exactly $n$, closed gaps of width $\theta$.   
   These boundary points form a closed set: take a sequence $x_n$ of them converging to 
   some $x$ and take $\alpha>\beta$ bigger than $x$ and each $x_n$.
   Take the $n$ disjoint strips of width $\theta$ in $U$ given by Lemma \ref{lemma:opengap}.
   Then infinitely many $x_n$ lie on the same boundary of one of the strips.
   Since the boundaries of these strips are closed in $M$, $x$ must be on the same boundary
   of the same strip.
\end{proof}

Let us say that $E\subset M$ is a {\em $(\theta,n)$-gap-club} if it is as in the conclusion of the previous lemma:
for each $\alpha\in\pi(F) = C$, $F\cap M_\alpha$ contains exactly the boundary points of $n$ closed gaps
of width $\theta$.
If $F$ is such a $(\theta,n)$-gap-club, for each $\alpha\in C$ there are 
$\phi_1^\alpha\ge\phi_2^\alpha\ge\dots\ge\phi_n^\alpha \ge 0$
which are the widths of the gaps between the closed gaps of width $\theta$ (see Figure \ref{fig:phi}).
Notice that we allow $\phi_i^\alpha$ to be $0$.
Let $\phi_i = \text{essinf}\phi_i^\alpha$, $i=1,\dots,n$.
\begin{figure}
   \begin{center}
       \epsfig{figure=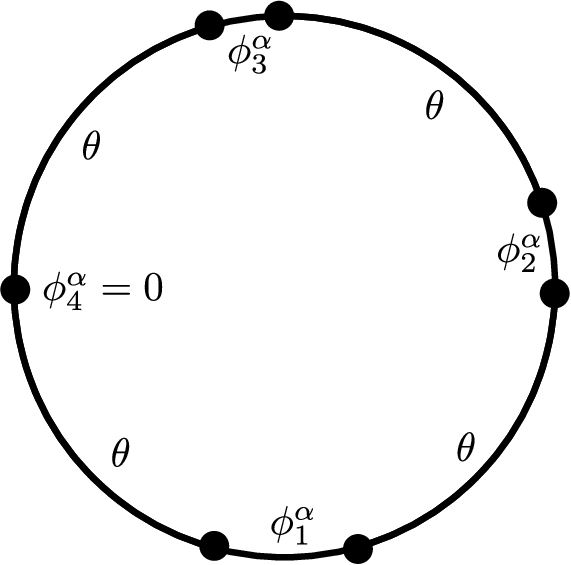, width=.35\textwidth}
       \caption{The inter-gaps $\phi_1^\alpha\ge\phi_2^\alpha\ge\dots\ge\phi_n^\alpha \ge 0$.
                The black dots are the members of $F\cap M_\alpha$.}
       \label{fig:phi}
   \end{center}
\end{figure}

\begin{lemma}
   \label{lemma:essinfphi_i}
   If $F\subset M$ is a $(\theta,n)$-gap-club, there is a club $D\subset\omega_1$ such that
   $\phi_i^\alpha = \phi_i$ for each $i=1,\dots,n$ and $\alpha\in D$.
\end{lemma}
\begin{proof}
   Essentially the same proof as that of Lemma \ref{lemma:essinf}.
   First, there is $\beta_0\in\omega_1$ such that $\phi_i^\alpha\ge\phi_i$ for each $i$
   and $\alpha\ge\beta_0$, by the same argument.
   Now, by Lemma \ref{lemma:opengap}, there is $\beta$ such that $M-F$ contains disjoint open gaps
   of width $\phi_i$, with $i=1,\dots,n$. These open gaps are also disjoint from the $n$ closed gaps
   of width $\theta$. It follows that if $\alpha\ge\beta$, $\alpha\in\pi(F)$,
   then $F\cap M_\alpha$ contains $n$ closed gaps of width $\theta$ and open gaps of width $\phi_1,\dots,\phi_n$
   (perhaps some equal to $0$). But for each $\alpha\in\pi(F)$, $n\theta + \sum_{i=1}^n\phi^\alpha_i =2\pi$, hence 
   in particular $n\theta + \sum_{i=1}^n\phi_i =2\pi$, so the boundary points of the closed gaps of width $\theta$
   are also the boundary points of these open gaps of width $\phi_1,\dots,\phi_n$.
\end{proof}

Let us now introduce a removal procedure.
Given a list $\mathcal{A}$ 
of distinct points $\varphi_1,\dots,\varphi_m\in\SSS^1$ written in increasing order, denote
by $\theta_i,\phi_i$ the widths of the gaps before and after $\varphi_i$.
If $\theta_i=\phi_i=\theta_j$ for each $i,j\le m$, we return the original list.
Otherwise, we remove all the points for which $\theta_i$ is minimal, and return
the remaining points re-enumerated as $\varphi_0,\dots,\varphi_k\in\SSS^1$ with $k<m$.
Denote the returned list by $r(\mathcal{A})$.
Notice that we always return at least one point.
After applying $r$ finitely many times, we end up with a list with constant gaps
between its members.

\begin{lemma}
   \label{lemma:final}
   If $F\subset M$ is a $(\theta,n)$-gap-club, there is a club $E\subset F$ such that
   $M_\alpha\cap E$ contains exactly $m$ points
   at distance $2\pi/m$, with $1\le m\le n$. I.e. $M_\alpha\cap E$ is
   made of $m$ closed gaps of width $2\pi/m$, for each $\alpha\in\pi(E)$.
\end{lemma}
\begin{proof}
   Let $D$ be given by the previous lemma. 
   Let $k$ be the number of indices $i$ such that $\phi_i>0$, then 
   each $M_\alpha\cap F$ contains exactly $m = n+k$ distinct points $\varphi^\alpha_1,\dots,\varphi^\alpha_m$
   for each $\alpha\in D$. Denote by $\phi^\alpha_i$ the width of the gap between 
   $\varphi^\alpha_i$ and $\varphi^\alpha_{i+1}$ (with $\varphi^\alpha_{m+1}=\varphi^\alpha_{1}$).
   Then $\phi^\alpha_i=\theta$ for $n$ indices. Denote this list of gaps by $\mathcal{A}_\alpha$.
   Notice that by Lemma \ref{lemma:x_ntheta_n}, 
   a converging sequence of points which all lie bewteen two gaps of sizes $\theta_1,\theta_2$
   has a limit which also lies bewteen two gaps of sizes $\theta_1,\theta_2$.
   Hence, 
   $\mathcal{A}_\alpha = \mathcal{A}_\beta$ up to a cyclic permutation
   whenever $\alpha,\beta\in D$. 
   Hence, if we remove all the points lying between gaps of some fixed sizes,
   what remains is still a closed set.
   Hence, the subspace given the union of all $r(\mathcal{A}_\alpha)$
   ($\alpha\in D$), is club. We may thus continue removing points with $r$ inductively until 
   having gaps that are all of equal size.
\end{proof}

\begin{proof}[{Proof of Theorem \ref{thm:Nyikos}}]
   Let $\theta = \text{essinf\,}h_F$. If $\theta = 0$, then (a) holds by Lemmas \ref{lemma:h_F=0}--\ref{lemma:essinf}.
   If $\theta>0$, let $C$ be club and $n\in\omega$ such that $h_F$ is constant with value $\theta$ 
   and $g_F$ is constant with value $n$
   on $C$. 
   Lemma \ref{lemma:ngaps} gives us a $(\theta,n)$-gap club $F_0\subset F$.
   Then, Lemma \ref{lemma:final} gives us the required $E\subset F_0$.   
\end{proof} 

\begin{defi}
   Let $M$ be a principal $\SSS^1$-bundle over $\LL_+$. If (a) of Theorem \ref{thm:Nyikos} holds, 
   we say that $M$ is of order $\infty$. If not, the order of $M$ is the minimal $n$
   such that (b) holds.
\end{defi}

\begin{lemma}
   \label{lemma:copyomega_1}
   $M$ contains a copy of $\omega_1$ iff $M$ is of order $1$.
\end{lemma}
\begin{proof}
   A club subset of $M$ which intersects $M_\alpha$ on exactly one point for a club set of $\alpha$
   contains a copy of $\omega_1$.
   Let $e:\omega_1\to M$ be an embedding.
   There is a club $C$ such that $\pi(e(\alpha))=\alpha$ for each $\alpha\in C$.
   Then $e\upharpoonright C$ is an embedding of a copy of $\omega_1$ which intersects each $M_\alpha$
   at most once.
\end{proof}
Notice that we do not claim that $M$ is homeomorphic
to the trivial bundle $\LL_+\times\SSS^1$.

\begin{lemma}
   \label{lemma:x_alpha}
   If $M$ is of order $\infty$ and $\{x_\alpha\,:\,\alpha\in\omega_1\}$ is an unbounded
   subset of $M$, there is a club $C$ such that $\wb{\{x_\beta\,:\,\beta<\alpha\}}\supset M_\alpha$
   when $\alpha\in C$.
\end{lemma}
\begin{proof}
   Let $\eta(\alpha)$ be minimal such that $\pi(x_{\eta(\alpha)})>\alpha$.
   Now, there is a club set $C_0$ such that $\pi(x_{\eta(\beta)})<\alpha$ for each $\beta<\alpha$ and $\alpha\in C_0$.
   Let $C=\pi(\wb{\{x_{\eta(\alpha)}\,:\,\alpha\in\omega_1\}})\cap C_0$.
   Then 
   $$
     \wb{\{x_\beta\,:\,\beta<\alpha\}}\cap M_\alpha \supset \wb{\{x_{\eta(\alpha)}\,:\,\alpha<\beta\}}\cap M_\alpha =
     \wb{\{x_{\eta(\alpha)}\,:\,\alpha\in\omega_1\}}\cap M_\alpha =M_\alpha
   $$ 
   for each $\alpha\in C$.
\end{proof}

\begin{lemma}
   \label{lemma:hcwn}
   If $M$ is of order $\infty$, then
   it is hereditarily collectionwise normal.
\end{lemma}
Recall that a family of subsets of a space $X$ is {\em discrete} iff each
point of $X$ has a neighborhood intersecting at most one member of the family.
The union of any discrete family of closed sets is closed.
A space is collectionwise Hausdorff [resp. collectionwise normal]
iff any discrete family of points [resp. closed sets] $\{D_\alpha\,:\,\alpha\in\kappa\}$
can be expanded to a disjoint family of open sets $U_\alpha\supset D_\alpha$.
\begin{proof}
   Let $N\subset M$ and
   $\mathscr{D}$ be a discrete-in-$N$ collection of closed-in-$N$ subsets.
   If $\cup\mathscr{D}$ has empty intersection with a club set of fibers $C$, we may expand 
   the members of $\mathscr{D}$ by disjoint open sets in
   each $\pi^{-1}((\alpha,\beta))$, where $\alpha,\beta$ are consecutive points of $C$,
   because $\pi^{-1}((\alpha,\beta))$ is a copy of $\R\times\SSS^1$.
   Suppose that $S = \pi(\cup\mathscr{D})$ is stationary.
   If $A,B\subset N$ are disjoint and closed-in-$N$, then at most one of them is unbounded.
   Indeed, 
   if both are there is some club $C_0$ such that be their closures in $M$ $\wb{A},\wb{B}$ both 
   contain $\pi^{-1}(D)$, and thus $M_\alpha\cap N\cap A\cap B\not=\varnothing$ when $\alpha\in S\cap C_0$. 
   Hence, if one member $D\in\mathscr{D}$ is unbounded, 
   then $\cup\mathscr{D}-D$ is (closed, hence) bounded and we may expand $\mathscr{D}$ as well.
   We now assume that each member of $\mathscr{D}$ is bounded.
   Note that $|\mathscr{D}|\le\aleph_1$ 
   since otherwise uncountably many intersect $\pi^{-1}((0,\alpha+1))$ for some $\alpha$,
   which is impossible for a discrete family, as $\R\times\SSS^1$
   is hereditarily Lindel\"of (for instance).
   If $|\mathscr{D}|=\aleph_0$, then $\pi(\cup\mathscr{D})$ is bounded
   and thus non-stationary, hence 
   $|\mathscr{D}|=\aleph_1$.
   We now show that it implies that
   $\mathscr{D}$ is actually not discrete.
   Denote its members by $D_\alpha$ for $\alpha\in\omega_1$
   and let $C_0$ be club such that $\cup_{\beta<\alpha}D_\beta\subset\pi^{-1}((0,\alpha))$
   for each $\alpha\in C_0$.
   Choosing one point in each $D_\alpha$ and applying Lemma \ref{lemma:x_alpha}
   we see that there is some $\alpha\in S\cap C_0$ such that 
   $M_\alpha\cap N\not=\varnothing$ is a subset of the closure in $N$
   of $\cup_{\beta<\alpha}D_\beta$.
   Hence, $\mathscr{D}$ is not a discrete family in $N$.
\end{proof}

\begin{lemma}
   \label{lemma:uncountable}
   If $M$ is of finite order $n$ and $E\subset M$ satisfies (b) of Theorem \ref{thm:Nyikos}, then 
   for each uncountable $E_0\subset E$ there is a club $D\subset\omega_1$
   such that $\wb{E_0}\supset E\cap\pi^{-1}(D)$.
\end{lemma}
\begin{proof}
   Otherwise, there is $m<n$ and a stationary set $S$ such that $\wb{E_0}\cap M_\alpha$
   contains $m$ points when $\alpha\in S$. By Lemma \ref{lemma:essinf},
   $h_{\wb{E_0}}$ is equal to $m$. 
   By Theorem \ref{thm:Nyikos} (b),
   there is a club subset of $\wb{E_0}$ which intersects each fiber in at most $m<n$ points,
   which is impossible since the order of $M$ is $n$.
\end{proof}

\begin{lemma}
   If $M$ is of finite order $n$, then $M$ is neither
   hereditarily normal nor hereditarily strongly collectionwise Hausdorff.
\end{lemma}
\begin{proof}
   Take a club subspace $E$ given by Theorem \ref{thm:Nyikos} (b) with $\pi(E)=C$.
   Set 
   $Y = E - E'$, that is, the subspace of isolated points of $E$.
   Then, $Y$ is closed discrete in $M-E'$.
   For any uncountable $Y_0\subset Y$, by Lemma \ref{lemma:uncountable}
   we have that (in $M$) $\wb{E_0}\supset M_\alpha\cap E$ for a club set $D$
   of $\alpha$.
   Suppose that for each $y\in Y$ there is $U_y\ni y$,
   then $U_y$ contains a (vertical) gap of width $1/m$ centered around $y$ for some $m$.
   Let $m$ be such that uncountably many $U_y$
   contain an open $\theta$-gap, with $\theta=\frac{1}{m}$.
   Let $Y_0$ be the subset of these $y$. 
   In particular, these $U_y$ contain the point $y+\theta/4$. 
   Let $F$ be the club set $\{x+\theta/4\,:\,x\in E\}$.
   Then the closure of $\{y+\theta/4\,:\,y\in Y_0\}$ is equal to $F\cap M_\alpha$ for a club set of $\alpha$.
   In particular, there are accumulation points of these $U_y$ in $M-E'$. It follows
   that the collection $U_y$ cannot be discrete in $M-E'$, while
   $Y$ is closed dicrete in it, that is: $M-E'$ is not strongly collectionwise Hausdorff.
   The same $Y$ together with the set 
   $\pi^{-1}(C)-E'$, which is
   closed in $M-E'$ and disjoint from $Y$, show that $M-E'$ is not hereditarily normal.
\end{proof}

We end this section with tangential remarks about an unrelated subject that
the non-interested reader may skip.
Say that a cover of a space $X$ by open sets $\mathscr{U}=\{U_\alpha\,:\,\alpha\in\omega_1\}$ 
is {\em systematic} iff $U_\alpha\not= X$ and $\wb{U_\alpha}\subset U_\beta$ whenever $\alpha<\omega_1$.
A closed subspace $Y\subset X$ is {\em narrow in $X$} iff 
for each 
systematic cover $\mathscr{U}$, either there is some $\alpha<\omega_1$ with $Y\subset U_\alpha$, or
$\wb{U_\alpha}\cap Y$ is Lindel\"of for each 
$\alpha<\omega_1$. Both $\omega_1$ (as a topological space) and $\LL_+$ are narrow in themselves, see 
\cite[Example 3.1]{Mesziguesnarrow}, and if $Y$ 
is narrow in itself, then any closed copy of $Y$ in a space $X$ is narrow in $X$.
By Theorem \ref{thm:ctblecomp}, under {\bf PFA}
any countably compact non-compact manifold contains (a closed copy of $\omega_1$ hence) 
a subspace which is narrow in it.
Nyikos claimed around 2006 that he has found a countably compact non-compact manifold that does not
contain any subspace narrow in it under $\diamondsuit^+$. Unfortunately, 
not even a draft supporting this claim is available, and we were unable to re-invent his construction
(see \cite[Section 4]{Mesziguesnarrow} for more on the subject). But we notice the following.

\begin{lemma}
   If $M$ is a principal $\SSS^1$-bundle over $\LL_+$, then $M$ is narrow in itself.
\end{lemma}
\begin{proof}
   Let $\{U_\alpha\,:\,\alpha\in\omega_1\}$ be a systematic cover. 
   We show that each $\wb{U_\alpha}$ is Lindel\"of, which implies that $M$ is narrow in itself.
   Suppose that some $\wb{U_\alpha}$ is non-Lindel\"of.
   Then $\pi(\wb{U_\alpha})$ contains a club subset $C$.
   Since $U_{\alpha+1}\supset\wb{U_\alpha}$,
   $U_{\alpha+1}\cap M_\beta$ contains a gap of size $\theta_{\gamma}>0$ for each $\gamma\in C$.
   There is some $n$ such that $\theta_{\gamma}\ge 1/n$ for a stationary set $S$ of $\gamma$.
   By Lemma \ref{lemma:opengap}, $U_{\alpha+1}\cap M_\gamma$ contains an open gap
   of width $\epsilon_\alpha = 1/n$.
   for each $\gamma$ above some $\beta_\alpha$.
   The same argument and induction shows that for each $\eta\ge\alpha+1$ there is some 
   $\epsilon_\eta>0$ and $\beta_\eta$ such that $U_\eta\cap M_\gamma$ contains an open $\epsilon_\eta$-gap
   for each $\gamma>\beta_\eta$, with the property that $\epsilon_\eta<\epsilon_\xi$ whenever $\alpha\ge\eta<\xi$.
   Hence, $\epsilon_\eta = 2\pi$ for some $\eta$. It follows that 
   $U_{\eta+1}\cap M_\gamma = M_\gamma$ for each $\gamma>\sup_{\xi\le\eta}\beta_\eta$. 
   \begin{claim}
      $U_{\eta+1}\supset\pi^{-1}((\xi,\omega_1))$ for some $\xi$.
   \end{claim}
   \begin{proof}
      If not, its complement is club. By Lemma \ref{lemma:club}, it intersects club-many $M_\gamma$, a contradiction.      
   \end{proof}
   Since $\pi^{-1}((0,\xi+1))$ is Lindel\"of, there is some $\gamma>\xi$ such that $U_\gamma$ contains
   all of $M$, contradicting the definition of systematic cover.
\end{proof}

%%%%%%%%%%%%%%%%%%%%%%%%%%%%%%%%%%%%%%%%%%%%%%%%%%%%%%%%%%%%%%%%%%%%%%%%%%%%%%%%%%%%%%%%%%%%%%%%%%%%%%%%%
%%%%%%%%%%%%%%%%%%%%%%%%%%%%%%%%%%%%%%%%%%%%%%%%%%%%%%%%%%%%%%%%%%%%%%%%%%%%%%%%%%%%%%%%%%%%%%%%%%%%%%%%%
%%%%%%%%%%%%%%%%%%%%%%%%%%%%%%%%%%%%%%%%%%%%%%%%%%%%%%%%%%%%%%%%%%%%%%%%%%%%%%%%%%%%%%%%%%%%%%%%%%%%%%%%%

\section{Tool: a collection of functions $z_\alpha:[0,\alpha]\to\omega$ from a ladder system}
\label{sec:tool}

In this section, we assume that we are given a ladder system $\mathscr{L}=\langle L_\alpha\,:\,\alpha\in\Lambda\rangle$
such that
\begin{equation}
  \label{eq:LalphaLambda0}
  %\begin{array}{l}
  L_\alpha\subset\Lambda\text{ when } \alpha\in\Lambda_2. %\text{ and } \\
  % L_\alpha =\{\beta + n\,:\,n\in\omega\}
  %\text{ when } \alpha = \beta + \omega\in\Lambda-\Lambda_2\text{, with }\beta\in\Lambda.
  %\end{array}
  %\tag{$\Lambda$}
\end{equation}

We denote the $k$-th member of $L_\alpha$ (in increasing order) by $\sigma_\alpha(k)$.
The goal of this section is to show the following (actually, its Corollary \ref{cor:z_alpha} below).
\begin{lemma}
  \label{lemma:z_alpha}
  Given a ladder system $\mathscr{L}$ satisfying (\ref{eq:LalphaLambda0}),
  there are functions $z_\alpha:[0,\alpha]\to\omega$ (interval taken in $\omega_1$) such that:
  \renewcommand\theenumi{\Roman{enumi}}
  \begin{enumerate}
     \item 
     \label{enum:lemma1}
     When $\eta<\alpha$, both in $\Lambda$, 
     there is $\beta<\eta$ such that for each $\gamma\in(\beta,\eta]$,
     $z_\alpha(\gamma) = z_\eta(\gamma)+z_\alpha(\eta)$.
     \item 
     \label{enum:lemma2}
     $z_\alpha(\alpha) = z_\alpha(\sigma_\alpha(k)) = 0$ for each $k\in\omega$.
     \item 
     \label{enum:lemma3}
     If $\alpha\in\Lambda_2$, for each $\eta\in\Lambda$,
     if there is a (maximal) $n\ge 0$ such that 
     $\sigma_\alpha(n)\le\eta<\alpha$,
     then
     there is $k\in\omega$ such that
     \begin{equation}
        \label{eq:z_alphaz_eta}
        \sigma_\alpha(n) <\sigma_\eta(k)<\eta
         \text{ and } 
        z_\alpha(\sigma_\eta(k))= z_\eta(\sigma_\eta(k)) + z_\alpha(\eta) + 1.
     \end{equation}   
     \item 
     \label{enum:lemma4}
     For each $\alpha\in\Lambda$, $z_{\alpha+\omega}$ is equal to $z_\alpha$ on $[0,\alpha]$
     and to $0$ on $[\alpha+1,\alpha+\omega]$, and $z_{\alpha+n}=z_{\alpha+\omega}\upharpoonright[0,\alpha+n]$
     for each $n\in\omega$.  
  \end{enumerate}
\end{lemma}

The idea behind this lemma is the following. (Caution: extremely vague notions approaching.) 
Imagine that you have a space of the form $X\times\LL_+$, or $X\times\omega_1$, where $X$ has a base point $x_0$
and a (vague) notion of `unique oriented distance'', that is: there is only one point at a
given distance from $x_0$. 
This is the case for $X=\SSS^1,\R,\Z$, the integers modulo $n$, etc.
Then $z_\alpha(\eta)$ indicates
at which ``positive distance from $\langle\eta,x_0\rangle$''  
does the ``center'' of a neighborhood $B$ of $\langle\alpha,x_0\rangle$ lie.
Hence, if one has $z_\alpha(\gamma)= z_\eta(\gamma) + z_\alpha(\eta)$ for $\gamma$ in some interval,
then it means that the neighborhoods of $\langle\alpha,x_0\rangle$
``follow'' those of $\langle\eta,z_\alpha(\eta)\rangle$ in this interval.
Lemma \ref{lemma:z_alpha} ensures that it is not the case in all of $[0,\eta]$,
and indeed the zone where they are appart (by $1$) is in some sense close to $\eta$,
viewed from $\alpha$.
\\
We first need the following (sub)lemma.
\begin{lemma}
    \label{lemma:fatladder}
    Let $\alpha\in\Lambda_2$ and $L_\alpha\subset\Lambda$ 
    be a ladder at $\alpha$.
    Then there exist a strictly increasing sequence $M_\alpha=\{\delta_\alpha(k)\,:\,k\in\omega\}\subset\Lambda$
    such that
    $L_\alpha\subset M_\alpha$
    (i.e. $\sigma_\alpha(n) = \delta_\alpha(k(n))$), 
    and if $\eta\in\Lambda$ and $\eta\ge\sigma_\alpha(0)$,
    then there is an $n>0$ such that either $\delta_\alpha(n) = \eta$, or
    $\delta_\alpha(n)\in\Lambda_2$ and
    $\sigma_{\delta_\alpha(n)}(0)<\eta<\delta_\alpha(n)$.
    %Moreover, for each $n$, if $\delta_\alpha(n)\in\Lambda_2$ there is a minimal $k\ge 0$ such that
    %$\sigma_{\delta_\alpha(n+1)}(k) \ge \delta_\alpha(n)$.
\end{lemma}
\begin{proof}
    We start with $M_\alpha=L_\alpha$ and 
    for each $n\ge 1$, we add finitely many members 
    to $M_\alpha$ between
    $\sigma_\alpha(n-1)$ and $\sigma_\alpha(n)$. We
    then re-enumerate the sequence when everything is defined.
    We proceed as follows. Fix $n\ge 1$.
    \renewcommand\theenumi{\arabic{enumi}}
    \begin{enumerate} 
    \item We let $\xi_0$ be $\sigma_\alpha(n)$.
    \item
       \label{enum:one}
       If there is no member of $\Lambda_2$ in the interval $(\sigma_\alpha(n-1),\xi_0]$,
       then $\xi_0 = \sigma_\alpha(n-1) + \omega\cdot m$ for some $m$ with $0<m<\omega$.
       We add all the intermediate $\sigma_\alpha(n-1)+\omega\cdot \ell$ for $\ell = 1,\dots,m-1$ in $M_\alpha$
       and set $\xi_1 = \sigma_\alpha(n)$.
    \item
       \label{enum:two}
       Else, we take the maximal member $\nu_0$ of $\Lambda_2$ in the interval $(\sigma_\alpha(n-1),\xi_0]$.
       If $\nu_0<\xi_0$, then $\xi_0 = \nu_0 + \omega\cdot m$ for some $m$ with $0<m<\omega$,
       we add all the intermediate $\nu_0+\omega\cdot \ell$ for $\ell = 0,\dots,m-1$
       (in particular, we add $\nu_0$) in $M_\alpha$.
       Then we define $\xi_1$ as follows.
       \begin{enumerate}
          \item
             If $\sigma_{\nu_0}(0) \le \sigma_\alpha(n-1)$, we let $\xi_1$ be equal to 
             $\sigma_{\nu_0}(k)$ for the smallest $k>0$ such that $\sigma_{\nu_0}(k) > \sigma_\alpha(n-1)$.
          \item
             If $\sigma_{\nu_0}(0) > \sigma_\alpha(n-1)$, we let $\xi_1$ be equal to $\sigma_{\nu_0}(0)$.
       \end{enumerate}
       We add the new $\xi_1$ to $M_\alpha$ and start again the procedure at point \ref{enum:one}, 
       replacing $\xi_0$ by $\xi_1$ (and defining $\nu_1$).
    \item
       We proceed by induction, adding points in $M_\alpha$ and defining $\xi_i$ and $\nu_i$,
       until $\xi_i = \sigma_\alpha(n)$ 
    \end{enumerate}
    Since $\xi_i$ decreases strictly with $i$, we add finitely many new members in $M_\alpha$.
    Doing it for each $n$, we obtain a new sequence where finitely many points
    have been added between $\sigma_\alpha(n)$ and $\sigma_\alpha(n+1)$, 
    $M_\alpha$ is thus again a sequence converging to $\alpha$.
    We may 
    thus
    enumerate $M_\alpha$ increasingly as $\delta_\alpha(k)$ ($k\in\omega$).
    The rest of the argument is clearer if we keep the algorithmic description above 
    instead of using
    the $\delta_\alpha(k)$ notation.
    If $\eta\in\Lambda$ is in the interval $(\sigma_\alpha(n-1),\sigma_\alpha(n)]$ for some $n\ge 1$,
    there is a maximal stage $i$ in the construction such that $\xi_i\ge\eta$.
    If we are in case \ref{enum:one}, then $\eta$ is one of the
    finitely many members of $\Lambda$ between $\sigma_\alpha(n)$ and $\xi_i$
    that were added at stage $i$, hence $\eta\in M_\alpha$.
    Otherwise, either $\nu_i\le\eta\le\xi_i$ and again $\eta\in M_\alpha$,
    or $\xi_{i+1}<\eta<\nu_i$.
    But in that case $\sigma_{\nu_i}(0)<\eta$, and since $\nu_i\in M_\alpha$, we are in the 
    second case we were looking for. %The ``moreover'' part is by construction.
\end{proof}

\begin{proof}[{Proof of Lemma \ref{lemma:z_alpha}.}]
    \ \\
    We define $z_\alpha$ by induction as follows.
    Set $z_0$ be constant on $0$.
    If $z_\alpha$ is defined for $\alpha\in\Lambda$, 
    $z_{\alpha+\omega}$ is defined by \ref{enum:lemma4}.
    Then \ref{enum:lemma1} and \ref{enum:lemma2} hold immediately.
    \\
    Let now $\alpha\in\Lambda_2$.
    Let $\delta_\alpha(n)$ be defined for each $n$ by Lemma \ref{lemma:fatladder}.
    We first set $z_\alpha(\alpha) = 0$. Fix $n\ge 0$ and
    take the smallest $k$ such that $\sigma_{\delta_\alpha(n+1)}(k) > \delta_\alpha(n)$.
    We define $z_\alpha\upharpoonright (\delta_\alpha(n),\delta_\alpha(n+1)]$ as
    \begin{equation}
       \label{eq:defz_alpha}
       z_\alpha(\beta)   = \left\{
         \begin{array}{ll}
             z_{\delta_\alpha(n+1)}(\beta) & \text{ if } \beta \in (\delta_\alpha(n), \sigma_{\delta_\alpha(n+1)}(k)]
                                                            \cup (\sigma_{\delta_\alpha(n+1)}(k+1),\delta_\alpha(n+1)] \\
             z_{\delta_\alpha(n+1)}(\beta) + 1 & \text{ if } \beta \in(\sigma_{\delta_\alpha(n+1)}(k),\sigma_{\delta_\alpha(n+1)}(k+1)]
         \end{array}
         \right.
    \end{equation}
    In particular, $z_{\alpha}(\delta_\alpha(n+1)) = z_{\delta_\alpha(n+1)}(\delta_\alpha(n+1)) = 0$ for all $n\ge 0$.\\
    We now prove that $z_\alpha$ has the desired properties by induction on $\alpha$.
    If $\alpha = \omega\cdot\omega$, then $z_\beta$ is identically $0$ for each $\beta<\alpha$, $\beta\in\Lambda$
    and $\delta_\alpha(n) = \omega\cdot n$.
    Then all the required properties follow from (\ref{eq:defz_alpha}).
    Let now $\alpha\in\Lambda$ and assume that the lemma holds for each $\beta<\alpha$. Then 
    \ref{enum:lemma2} follows directly by (\ref{eq:defz_alpha}). 
    If $\alpha\not\in\Lambda_2$, $\alpha=\beta+\omega$ for some $\beta\in\Lambda$
    and \ref{enum:lemma1} follows by definition of $z_\alpha$.
    We thus assume that $\alpha\in\Lambda_2$.
    Let $\eta\in\Lambda$, $\eta<\alpha$. Let $n$ be minimal such that $\delta_\alpha(n)\le\eta$.
    Then $z_\alpha(\delta_\alpha(n)) = z_{\delta_\alpha(n)}(\delta_\alpha(n))=0$. 
    By induction \ref{enum:lemma1} holds for $\delta_\alpha(n)$,
    hence 
    there is some $\beta<\eta$ such that
    $$ z_{\delta_\alpha(n)}(\gamma) = z_{\delta_\alpha(n)}(\eta) + z_{\eta}(\gamma)$$
    for each $\gamma\in(\beta,\eta]$. By (\ref{eq:defz_alpha}), up to taking a bigger $\beta$ if necessary,
    both $\beta$ and $\eta$ belong to an interval in which either $z_{\delta_\alpha(n)}$ and $z_\alpha$
    take the same values, or in which $z_\alpha$ takes the values of $z_{\delta_\alpha(n)}$ plus $1$.
    In the former case, \ref{enum:lemma1} follows immediately,
    in the latter it is enough to compute
    $$ 
      z_{\alpha}(\gamma) = z_{\delta_\alpha(n)}(\gamma) + 1 = z_{\delta_\alpha(n)}(\eta) +1 + z_{\eta}(\gamma)
      = z_{\alpha}(\eta) + z_{\eta}(\gamma)
    $$
    for each $\gamma\in(\beta,\eta]$. This proves \ref{enum:lemma1}.
    We now prove \ref{enum:lemma3} and thus suppose that 
    $\sigma_\alpha(0)<\eta<\alpha$.
    By Lemma \ref{lemma:fatladder}, there is some $n\ge 0$ such that
    either $\eta=\delta_\alpha(n+1)$, or $\sigma_{\delta_\alpha(n+1)}(0) < \eta < \delta_\alpha(n+1)$ 
    and $\delta_\alpha(n+1)\in\Lambda_2$.
    In the former case, $z_\alpha(\eta) = 0$ and by (\ref{eq:defz_alpha}) 
    we know that $z_\alpha(\sigma_\eta(k)) = z_{\eta}(\sigma_\eta(k))+1$.
    Since $\sigma_\eta(k) > \delta_\alpha(n)$ and $L_\alpha$ is a subsequence of $M_\alpha$,
    we know in particular that if $\ell$ is maximal such that $\sigma_\alpha(\ell)<\eta$
    then $\sigma_\eta(k) > \sigma_\alpha(\ell)$ as well. Hence, \ref{enum:lemma3} holds in that case.
    We now assume that $\sigma_{\delta_\alpha(n+1)}(0) < \eta < \delta_\alpha(n+1)$ and $\delta_\alpha(n+1)\in\Lambda_2$.
    By induction, if $k$ is maximal such that $\sigma_{\delta_\alpha(n+1)}(k) < \eta$,
    there is some $\ell$ such that 
    \begin{equation}
       \label{eq:f_alphainduction}
       z_{\delta_\alpha(n+1)}(\sigma_\eta(\ell)) = z_\eta(\sigma_\eta(\ell)) + z_{\delta_\alpha(n+1)}(\eta) + 1.
    \end{equation}
    and $\sigma_{\delta_\alpha(n+1)}(k)<\sigma_\eta(\ell) <\eta$.
    Then, by (\ref{eq:defz_alpha}), either $k$ is also minimal such that
    $\sigma_{\delta_\alpha(n+1)}(k) > \delta_\alpha(n)$, 
    in which case $z_\alpha(\sigma_\eta(\ell)) = z_{\delta_\alpha(n+1)}(\sigma_\eta(\ell)) + 1$
    and $z_\alpha(\eta) = z_{\delta_\alpha(n+1)}(\eta) + 1$, or $k$ is not minimal, in which case
    $z_\alpha(\sigma_\eta(\ell)) = z_{\delta_\alpha(n+1)}(\sigma_\eta(\ell))$
    and $z_\alpha(\eta) = z_{\delta_\alpha(n+1)}(\eta)$.
    That is, if one replaces $\delta_\alpha(n+1)$ by $\alpha$ in 
    (\ref{eq:f_alphainduction}), then either we add $1$ on both sides of the equation, or the equation stays the same. 
    In both cases, \ref{enum:lemma3} holds.
    %This shows the required properties of $z_\alpha$.
\end{proof}

\begin{cor}
   \label{cor:z_alpha}
   If $z_\alpha$ is defined as in Lemma \ref{lemma:z_alpha}, then for each $\alpha\in\Lambda_2$
   and each increasing sequence $\beta_n\in\Lambda$ with supremum $\alpha$, there are $\xi_n\in L_{\beta_n}$
   such that the sequence $\xi_n$ has limit $\alpha$ and 
   $z_\alpha(\xi_n) = z_{\beta_n}(\xi_n) + z_\alpha(\beta_n) + 1$.
\end{cor} 
\begin{proof}
   This is immediate since if one takes the biggest $m$ such that
   $\sigma_\alpha(m)<\beta_n$, Lemma \ref{lemma:z_alpha} provides us such a $\xi_n$
   which is above $\sigma_\alpha(m)$.
\end{proof}

We end this section with a definition for further reference.
\begin{defi}
   \label{defi:modelled}
   Let $\mathscr{L}$ be a ladder system satisfying (\ref{eq:LalphaLambda0}).
   We say that a family $z_\alpha$ of functions 
   $[0,\alpha]\to\omega$ 
   is {\em modelled on $\mathscr{L}$} if
   it satisfies \ref{enum:lemma1}--\ref{enum:lemma2}--\ref{enum:lemma3}--\ref{enum:lemma4}
   of Lemma \ref{lemma:z_alpha} and also (\ref{eq:defz_alpha}), 
   where the ladder system $\delta_\alpha(n)$ is
   given by Lemma \ref{lemma:fatladder}.
\end{defi}

Notice in passing that the ladder system $\delta_\alpha(n)$ and the functions $z_\alpha$
depend only on $\mathscr{L}$, as no choices were made in their definitions.

%%%%%%%%%%%%%%%%%%%%%%%%%%%%%%%%%%%%%%%%%%%%%%%%%%%%%%%%%%%%%%%%%%%%%%%%%%%%%%%%%%%%%%%%%%%%%%
%%%%%%%%%%%%%%%%%%%%%%%%%%%%%%%%%%%%%%%%%%%%%%%%%%%%%%%%%%%%%%%%%%%%%%%%%%%%%%%%%%%%%%%%%%%%%%
%%%%%%%%%%%%%%%%%%%%%%%%%%%%%%%%%%%%%%%%%%%%%%%%%%%%%%%%%%%%%%%%%%%%%%%%%%%%%%%%%%%%%%%%%%%%%%
\section{Constructing bundles from a ladder system}
\label{sec:constr}

We now return to our bundles.
Combining the features of angle preservation and the fact that
$f_\alpha$'s are isometries on the fibers gives us a convenient way of defining a local base at
any point $x$. Fix a choice of $f_\alpha : (0,\alpha+1) \times\SSS^1\to\pi^{-1}((0,\alpha+1))$
 for each $\alpha\in\omega_1$ , and for
$p = \alpha+ r$ ($\alpha\in\omega_1$, $r \in [0, 1)$) 
let $f_p = f_\alpha$, as above.
Then the topology in $\pi^{-1}((0,\alpha+1))$ is entirely determined by the strips
$S_\alpha(p,q,\phi,\theta)$ (Definition \ref{defi:strips})
with $0<p<q<\alpha+1$ and $\phi,\theta\in\SSS^1$.
The horizontal boundaries of these strips are parts of $f_\alpha\bigl((0,\alpha+1)\times\{\phi\}\bigr)$
and $f_\alpha\bigl((0,\alpha+1)\times\{\phi+\theta\}\bigr)$.
Let $a_\alpha:(0,\alpha+1)\to\SSS^1$ be defined as $a_\alpha(p) = f_\alpha(\langle p,0\rangle)$,
then $f_\alpha(\langle p,\theta\rangle) = a_\alpha(p) + \theta$.
Hence, the ``arcs'' $a_\alpha$ determine entirely the topology.
We say that the collection $a_\alpha$ engender the topology.
Of course, these arcs must behave coherently with each other (in some sense) in order to
obtain a manifold topology. A criteria is given later.
In the trivial bundle, $a_\alpha(p) = 0$ for each $\alpha$, that is:
each $a_\alpha$ is ``completely flat''. 
Nyikos' idea is to introduce ``bumps'' in $a_\alpha$ so that these bumps accumulate on a club set of $M_\alpha$,
which force club subsets of $M$ to have ``too many points'' in club-many fibers.
This will be done by defining
these $a_\alpha$ as  extensions to $\LL_+$ of the function $\theta\cdot z_\alpha\text{(mod }2\pi\text{)}$,
where $z_\alpha$ is modelled on $\mathscr{L}$ (Definition \ref{defi:modelled}) and $\theta\in(0,2\pi)$. 
The purpose of this section is to make these vague claims explicit.
It is important to note that the functions $a_\alpha$ will in general {\em not} be continuous.

%%%%%%%%%%%%%%%%%%%%%%%%%%%%%%%%%%%%%%%%%%%%%%%%%%%%%%%%%%%%%%%%%%%%%%%%%%%%%%%%%%%%%%%%%%%%%
\subsection{Properties of $a_\alpha$ that yield a bundle topology}

Let us now be a bit more rigorous and define properly the tools we are going to use.
The underlying set of our bundles is always $\LL_+\times\SSS^1$, and we will define
the topology by induction on $(0,\alpha+1)\times\SSS^1$, which we denote by $R_\alpha$,
and $R_{\omega_1}$ denotes $\LL_+\times\SSS^1$.
To avoid confusion, we denote by $\mu_\alpha$ the usual product topology
on $R_\alpha$, and by $\tau_\alpha$ our new topology (to be defined).
We let also $\pi:\LL_+\times\SSS^1\to\LL_+$ be the projection on the first factor and
as before denote $\pi^{-1}(\{p\})$ by $M_p$.
If $\rho$ is a topology on some set $X$ and $Y\subset X$, we denote by $\rho\upharpoonright Y$
the topology induced on $Y$ by $\rho$.
In our constructions, we will always ensure that 
\begin{align}
   \label{eq:restriction}
   %\tag{$\star$}
      \beta\le\alpha \quad\Longrightarrow\quad \tau_\alpha\upharpoonright R_\beta = \tau_\beta,\\
   \label{eq:picont}
   %\tag{$\star\star$}
      \pi\upharpoonright R_\alpha\text{ is $\tau_\alpha$-continuous.}
\end{align}

Instead of defining $\tau_\alpha$ directly, we define our functions $a_\alpha$,
and first look for conditions under which (\ref{eq:restriction}) 
and (\ref{eq:picont}) hold for the engendered topology.

\begin{defi}
   \label{defi:g_s}
   Given a function (not necessarily continuous) $s:\LL_+\to\SSS^1$,
   we let 
   $\chi_{s}:R_{\omega_1}\to R_{\omega_1}$ be defined as  
   $$ \chi_{s}(\langle p,\theta\rangle) = \langle p,s(p) + \theta\rangle.$$
\end{defi}
Of course, if $s(p)=0$, then $\chi_s\upharpoonright M_p$
is the identity on $M_p$.

\begin{lemma}
   \label{lemma:chi_commute}
   If $r,s:\LL_+\to\SSS^1$ are two (not necessarily continuous) functions, then
   $\chi_r\circ\chi_s = \chi_s\circ\chi_r = \chi_{r+s}$.
\end{lemma}
\begin{proof}
   Immediate.
\end{proof}

Recall that the support $\text{supp}(f)$ of a function $f$
is those $x$ for which $f(x)\not= 0$.
It will be convenient to extend the domain of $a_\alpha$ to all
of $\LL_+$, but with support a subset of $(0,\alpha)$, that is:
\begin{equation}
   \label{eq:a_alpha_0}
   %\tag{$a_\alpha$}
   a_\alpha(p) = 0\,\forall p\ge\alpha,\text{ i.e. supp}(a_\alpha)\subset(0,\alpha).
\end{equation}
Given such $a_\alpha$, we write $f_\alpha$ for $\chi_{a_\alpha}\upharpoonright R_\alpha$ defined as above.
\begin{defi}
   \label{defi:tau_alpha}
   Given $a_\alpha, f_\alpha$ as above,
   we set $\tau_\alpha = \{f_\alpha(U)\,:\,U\in\mu_\alpha\}$. 
\end{defi}
\begin{lemma}
   \label{lemma:tau_alpha}
   Assume that $a_\alpha$, $f_\alpha$ and $\tau_\alpha$ are as above. Then the following holds.\\
   (a) $\tau_\alpha$ is indeed a topology whose basis is given by the images by $f_\alpha$ of all the strips 
       $(a,b)\times(\phi,\phi+\theta)\subset R_\alpha$, that is, by the
       $S_\alpha(a,b,\phi,\theta)$. \\
   (b) $f_\alpha:\langle R_\alpha,\mu_\alpha\rangle\to\langle R_\alpha,\tau_\alpha\rangle$
       is a homeomorphism.\\
   (c) (\ref{eq:picont}) holds, that is: $\pi\upharpoonright R_\alpha$ is $\tau_\alpha$-continuous.
\end{lemma}
\begin{proof}
   \
   \\
   (a) Since $f_\alpha$ is a bijection, $\tau_\alpha$ is indeed a topology.
       If $\mathscr{U}$ is a collection of subsets of $R_\alpha$, then 
       $f_\alpha(\cup\mathscr{U}) = \cup\{f(U)\,:\,U\in\mathscr{U}\}$.
       The result follows since the strips form a basis for $\mu_\alpha$.\\
   (b) It follows directly from the definition and the fact that $f_\alpha$ is a bijection.\\
   (c) If $\langle p_n,\theta_n\rangle$ $\tau_\alpha$-converge to $\langle p,\theta\rangle$,
       then $p_n$ converges to $p$ in the usual topology of $(0,\alpha+1)$. 
\end{proof}

The next two lemmas will be useful when we investigate cluster points of infinite sets in our bundles.
The proof of the first one is much shorter than its statement.
\begin{lemma}
   \label{lemma:converge}
   Let $\alpha\in\Lambda$ and $a_\alpha$ be as above. 
   Suppose that for some $A$ cofinal in $[0,\alpha)$ (interval in $\omega_1$), there
   is $\theta_\beta$ for each $\beta\in A$ such that the $\tau_\alpha$-closure of
   $\{\langle \beta,\theta_\beta\rangle\,:\,\beta\in A\}$ intersected with $M_\alpha$
   contains only $\langle\alpha,\theta\rangle$.
   Then for each $\epsilon>0$ there is $\gamma<\alpha$ such that 
   $$ 
     |a_\alpha(\gamma) +\theta- \theta_\beta| \le \epsilon
     \quad \forall \beta\ge \gamma,\, \beta\in A.
   $$
\end{lemma}
\begin{proof}
   By (a) of Lemma \ref{lemma:tau_alpha}.
\end{proof}

Let $\alpha\in\Lambda$ and $\beta<\alpha<p<\alpha+1$, $\epsilon >0$.
Assume that $a_\alpha$ is defined as above. 
For $H\subset M_\alpha$ we set $T_\alpha(\beta,H,\epsilon)$ to be the union of the strips 
$S_\alpha(\beta,\alpha,\phi-\frac{\epsilon}{2},\epsilon)$ 
for each $\phi$ such that $\langle \alpha,\phi\rangle\in H$. In words: 
the union of the strips of width $\epsilon$
centered around points of $M_\alpha\cap H$.

\begin{lemma} 
   \label{lemma:Fclosed} 
   Let $\alpha\in\Lambda$ and $a_\alpha$ be as above.
   Let $F\subset R_\alpha$ be $\tau_\alpha$-closed. 
   Then for each $\epsilon > 0$ there is $\beta<\alpha$
   such that $F\cap(\beta,\alpha)\subset T_\alpha(\beta,H,\epsilon)$.
\end{lemma}
\begin{proof}
   Otherwise, for some $\epsilon>0$ there are points of $F$ outside of
   $T_\alpha(\beta,H,\epsilon)\cap(\beta,\alpha)$ for each $\beta<\alpha$.
   By the previous lemma and countable compactness, there must be a point of $F$ 
   in $M_\alpha$ which is at distance $>\epsilon/2$ of $F\cap M_\alpha$, a contradiction.
\end{proof}

\begin{defi}
   \label{defi:ring-homeomorphism}
   We call a function $h:R_\alpha\to R_\alpha$ a ring-homeomorphism iff 
   $h=\chi_s\upharpoonright R_\alpha$ for a continuous $s:\LL_+\to\SSS^1$.
\end{defi}
Hence, $f_\alpha$ is {\em almost} a ring-homeomorphism, except for the fact that $a_\alpha$ is not assumed
(and will not be) continuous. Ring-homeomorphisms are indeed homeomorphisms for $\mu_\alpha$ and $\tau_\alpha$:

\begin{lemma}
   \label{lemma:chihomeo}
   Let $\chi_s$ be a ring-homeomorphism of $R_\alpha$. Then the diagram below commutes and each map
   is a homeomorphism.
   \begin{center} 
     \begin{tikzcd}
     \langle R_\alpha,\mu_\alpha\rangle 
       \arrow[r, rightarrow , "\chi_s"]
       \arrow[d, rightarrow , "f_\alpha"]
     &
     \langle R_\alpha,\mu_\alpha\rangle 
       \arrow[d, rightarrow , "f_\alpha"]
     \\
     \langle R_\alpha,\tau_\alpha\rangle 
       \arrow[r, rightarrow , "\chi_s"]
     &
     \langle R_\alpha,\tau_\alpha\rangle
     \end{tikzcd}
   \end{center}
\end{lemma}
\begin{proof}
   It should be clear that $\chi_s$ is a $\mu_\alpha$-homeomorphism since $s$ is assumed to be continuous.
   Since $f_\alpha\circ\chi_s = \chi_s\circ f_\alpha$ by Lemma \ref{lemma:chi_commute},
   $\chi_s = f_\alpha\circ \chi_s \circ f_\alpha^{-1}$ is also a $\tau_\alpha$-homeomorphism.
\end{proof}

\begin{defi}
   Let $a,b:\LL_+\to\SSS^1$.
   We write $ a =^*_\beta b $
   iff there is a continuous $s:\LL_+\to\SSS^1$ such that 
   $(a+s)\upharpoonright(0,\beta+1) = b\upharpoonright(0,\beta+1)$.    
\end{defi}
Notice that we may ask $s$ to be a continuous function with domain $(0,\beta+1)$ instead of $\LL_+$
without altering the definition.

\begin{lemma}
   \label{lemma:=^*}
   Let $\gamma\le\beta\le\alpha$, $a,b,c:\LL_+\to\SSS^1$, and $a_\alpha$ be defined as above.
   \
   \\
   (a) If $c =^*_\gamma b =^*_\beta a$, then $c =^*_\gamma a$.\\
   (b) If $a_\alpha =^*_\beta a_\beta$, then (\ref{eq:restriction}) holds, that is:
       $\tau_\alpha\upharpoonright R_\beta = \tau_\beta$.
\end{lemma}
\begin{proof}
   \ \\
   (a) Writing the restrictions is a bit tedious, 
       but in essence the result is immediate since $s+r$ is continuous whenever
       $r,s:\LL_+\to\SSS^1$ are continuous.\\
   (b)
   Set $\wt{a}_\beta = a_\alpha\upharpoonright (0,\beta+1)$ and $\wt{f}_\beta = \chi_{\wt{a}_\beta}$.
   Then 
   $$ \wt{f}_\beta = \chi_{\wt{a}_\beta} = \chi_{a_\beta + s} = \chi_{a_\beta}\circ\chi_s $$
   by Lemma \ref{lemma:chi_commute}, and the result follows since $\chi_s$ is a $\tau_\beta$-homeomorphism
   by Lemma \ref{lemma:chihomeo}.
\end{proof}

Let us summarize what we showed in a theorem.
\begin{thm}
   \label{thm:Mbundle}
   Let $a_\alpha,f_\alpha,\tau_\alpha$ be defined 
   as above
   for each $\alpha\in\omega_1$
   such that $a_\beta =^*_\beta a_\alpha$ for each $\beta\le\alpha$.
   Then (\ref{eq:restriction}) and (\ref{eq:picont}) hold, and 
   $M = \LL_+\times\SSS^1$ with topology $\tau$ whose basis is 
   $\cup_{\alpha<\omega_1}\tau_\alpha$
   is a principal $\SSS^1$-bundle over $\LL_+$. In particular,
   $\pi:M\to\LL_+$ is $\tau$-continous.
\end{thm}
\begin{proof}
   Lemmas \ref{lemma:tau_alpha} and \ref{lemma:=^*} imply that (\ref{eq:restriction}) and (\ref{eq:picont}) hold.
   The continuity of 
   the action of $\SSS^1$ by rotations is ensured in each $R_\alpha$ by 
   construction, since it sends strips to strips.
\end{proof}

%%%%%%%%%%%%%%%%%%%%%%%%%%%%%%%%%%%%%%%%%%%%%%%%%%%%%%%%%%%%%%%%%%%%%%%%%%%%%%%%%%%%
\subsection{Defining $a_\alpha$ from $z_\alpha$}

Suppose that we have a family of functions $z_\alpha:[0,\alpha]\to\omega$ modelled on a ladder system $\mathscr{L}$
(Definition \ref{defi:modelled}).
We extend the domain of $z_\alpha$ to all of $\omega_1$ by taking the value $0$ above $\alpha$.
Fix some $\theta\in[0,2\pi)$.
Then, we define $a_\alpha$ as follows.
First, it takes the same values as $\theta\cdot z_\alpha\text{ (mod }2\pi)$
on ordinals.
If $z_\alpha(\beta) = z_\alpha(\beta+1)$, then $a_\alpha$ takes the constant value $z_\alpha(\beta)$
on $[\beta,\beta+1]$ (interval in $\LL_+$, of course).
If $z_\alpha(\beta) \not= z_\alpha(\beta+1)$, we interpolate continuously (for instance linearly).
Then (\ref{eq:a_alpha_0}) holds by condition \ref{enum:lemma2} of Lemma \ref{lemma:z_alpha}.
\\
We now show that $a_\alpha =^*_\beta a_\beta$ for each $\beta<\alpha$ in $\omega_1$, by induction on $\alpha$.
By \ref{enum:lemma4} of Lemma \ref{lemma:z_alpha}, if it holds for $\alpha$, then it holds for $\alpha+\omega$.
To show the case $\alpha\in\Lambda_2$, notice that
by (\ref{eq:defz_alpha}), 
$a_\alpha$ is related to (or obtained from) the $a_\eta$ with smaller $\eta$
as follows. There are sequences $\xi_n<\gamma_n<\beta_{n}<\xi_{n+1}$ with limit $\alpha$
($\xi_n$ being actually $\delta_\alpha(n)$), such that
\begin{itemize}
   \item[(a)] $a_\alpha(\beta_n) = a_{\beta_n}(\beta_n)=0$;
   \item[(b)] $a_\alpha(p) = a_{\beta_n}(p)$ if $p\in [\beta_{n-1}+1,\xi_n]\cup [\gamma_n+1,\beta_n]$;
   \item[(c)] $a_\alpha(p) = a_{\beta_n}(p) + \theta$ if $p\in[\xi_n+1,\gamma_n]$;
   \item[(d)] $a_\alpha$ is continuous on each interval $[\eta,\eta+1]$ with $\eta<\alpha$.
\end{itemize}
(The last point holds by definition of $a_\alpha$.) 
The next lemma shows that $a_\alpha =^*_\eta a_\eta$ for each $\eta<\alpha$.

\begin{lemma}
   \label{lemma:crux}
   Let $\alpha\in\Lambda$. Assume that $a_\eta:\LL_+\to\SSS^1$ is defined
   for each $\eta<\alpha$, 
   with $\text{supp}(a_\eta)\subset(0,\eta)$,
   such that for each $\gamma\le\eta<\alpha$ we have $a_\gamma =^*_\gamma a_\eta$.
   Let $\beta_n,\gamma_n,\xi_n$ be sequences of ordinals with limit $\alpha$ such that 
   $\xi_n<\gamma_n<\beta_{n}<\xi_{n+1}$
   and let $\theta\in(-2\pi,2\pi)$ such that (a), (b), (c), (d) just above hold.
   Then 
   $a_\alpha =^*_\eta a_\eta$
   for each $\eta<\alpha$. 
\end{lemma}

We first show the following lemma.
We believe that it is convenient to introduce step functions.
If $\gamma<\beta$ in $\omega_1$, we define a {\em $\beta$-step} to be any continuous
function $\LL_+\to[0,1]$ which is constant with value $0$ on $[0,\beta]$ and constant with value $1$ 
in $[\beta+1,\omega_1)$. A {\em $(\gamma,\beta)$-bump} is the difference $c-d$, where $c$ is a
$\gamma$-step and $d$ a $\beta$-step.
In words: a $(\gamma,\beta)$-bump takes value $1$ in $[\gamma+1,\beta]$
and $0$ in $(0,\gamma]\cup[\beta+1,\omega_1)$. We also call $\beta$-step and $(\gamma,\beta)$-bumps
the restriction of these functions to $(0,\alpha+1)$ for any $\alpha>\beta>\gamma$.

\begin{lemma}
   \label{lemma:abcd}
   Let $\gamma<\beta$, $a,b:\LL_+\to\SSS^1$ be (not necessarily continuous) functions
   with $\text{supp}(a)\subset (0,\gamma]$, $\text{supp}(b)\subset (0,\beta]$.
   Let
   $u$ be a $\gamma$-step and set
   $d = u\cdot b + (1-u)\cdot a$. If $a =^*_\gamma b$, then 
   $d =^*_\beta b$.
\end{lemma}

Here, the multiplication is again understood modulo $2\pi$ in $[0,2\pi)$.

\begin{proof}
   Since $\text{supp}(a)\subset (0,\gamma]$, $a\cdot (1-u) = a$. Let $s$ be continuous such that 
   $a+s = b$ on $(0,\gamma+1)$. Since $(1-u)$ has support in $(0,\gamma+1)$,
   $$ b = u\cdot b + (1-u)\cdot b = u\cdot b + (1-u)\cdot(a+s) = u\cdot b + a + (1-u)\cdot s.$$
   Hence,
   $$ d = u\cdot b + (1-u)\cdot a = b - a - (1-u)\cdot s + a = b - (1-u)\cdot s,$$
   which finishes the proof since $(1-u)\cdot s$ is continuous.
\end{proof}

\begin{proof}[Proof of Lemma \ref{lemma:crux}]
   By (a), (b), (c), 
   $$ a_\alpha = \sum_{n\in\omega} \theta\cdot( b_n\cdot a_{\beta_n} + d_n) $$
   where $b_n$ is a $(\beta_{n-1},\beta_n)$-bump and $d_n$ a $(\xi_{n},\gamma_n)$-bump.
   To show that $a_\alpha =^*_\gamma a_\gamma$ for each $\gamma<\alpha$,
   it is enough to show that $a_\alpha =^*_\beta a_{\beta_n}$ for each $n$ by Lemma \ref{lemma:=^*} (a)
   and the assumption $a_{\beta_n} =^*_\gamma a_\gamma$ for each $\gamma<\beta_n$.  
   We do it by induction on $n$. It is clear for $n=0$ since 
   $a_\alpha\upharpoonright (0,\beta_0+1) = a_{\beta_0} + d_0$ and $d_0$ is continuous.
   For the induction step, let $a = \sum_{n\le m} b_n\cdot a_{\beta_n}$, 
   $b = b_{m+1}\cdot a_{\beta_{m+1}}$.
   Then we have $a_\alpha\upharpoonright(0,\beta+1) = u\cdot a + (1-u)\cdot b$
   where $u$ is a $\beta_n$-step,
   and the result follows by Lemma \ref{lemma:abcd}.
\end{proof}
   
Notice that the above construction depends only on the ladder system $\mathscr{L}$ and $\theta$.
Write $M(\mathscr{L},\theta)$ for the resulting space.

\begin{cor}
   \label{cor:Mbundle}
   If $a_\alpha$ is defined as in this subsection, then $\pi:M(\mathscr{L},\theta)\to\LL_+$
   (with topology in $M(\mathscr{L},\theta)$ engendered by $a_\alpha$)
   is a principal $\SSS^1$-bundle over $\LL_+$.
\end{cor}
\begin{proof}
   By Theorem \ref{thm:Mbundle} and Lemma \ref{lemma:crux}.
\end{proof}

%%%%%%%%%%%%%%%%%%%%%%%%%%%%%%%%%%%%%%%%%%%%%%%%%%%%%%%%%%%%%%%%%%%%%%%%%%%%%%%%%%%%%%%%%%%%%%%%%%%%%%%%%  
%%%%%%%%%%%%%%%%%%%%%%%%%%%%%%%%%%%%%%%%%%%%%%%%%%%%%%%%%%%%%%%%%%%%%%%%%%%%%%%%%%%%%%%%%%%%%%%%%%%%%%%%%
%%%%%%%%%%%%%%%%%%%%%%%%%%%%%%%%%%%%%%%%%%%%%%%%%%%%%%%%%%%%%%%%%%%%%%%%%%%%%%%%%%%%%%%%%%%%%%%%%%%%%%%%%

\section{Bundles of any order with a $\clubsuit_C$-sequence}
\label{sec:usingclub}

In this section, we put together the results of the previous two.
%Recall that if $\alpha\in\Lambda-\Lambda_2$, then $\alpha=\beta+\omega$ for some $\beta\in\Lambda\cup\{0\}$.

\begin{rem}
  A $\clubsuit_C$-sequence $\mathscr{L}=\langle L_\alpha\,:\,\alpha\in\Lambda\rangle$ can be assumed to satisfy
  (\ref{eq:LalphaLambda0}) without loss of generality.
\end{rem}

Given $M=M(\mathscr{L},\theta)$ defined as in the previous section and $\gamma\le\beta\le\alpha$, we define the jolt of 
$\alpha,\beta$ at $\gamma$ as
$$ \jolt_{\alpha,\beta}(\gamma) = a_\alpha(\gamma) - a_\beta(\gamma) - a_\alpha(\beta) .$$
In words: since $a_\beta(\beta) = 0$,
we push $a_\beta$ up by $a_\alpha(\beta)$, which gives a function $\wt{a}_\beta$ such that
$\wt{a}_\beta(\beta) = a_\alpha(\beta)$. Then
$\jolt_{\alpha,\beta}(\gamma)$ is the difference between $\wt{a}_\beta$ and $a_\alpha$
at the point $\gamma$, see Figure \ref{fig:jolt}.
\begin{figure}
   \begin{center}
       \epsfig{figure=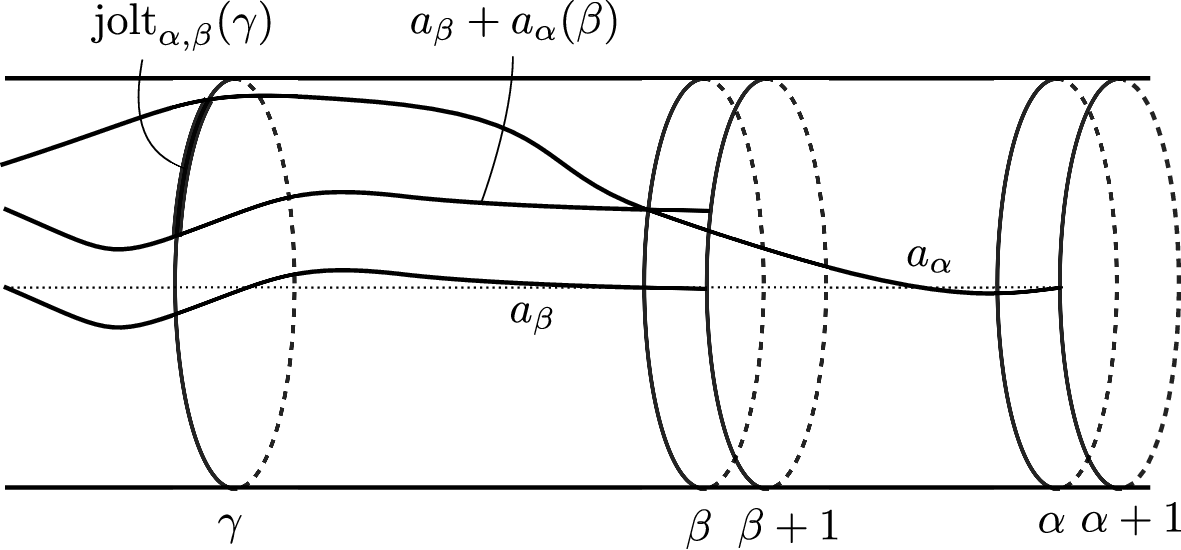, width=.6\textwidth}
       \caption{$\jolt_{\alpha,\beta}(\gamma)$}
       \label{fig:jolt}
   \end{center}
\end{figure}

The next lemma is the crucial part of Nyikos' construction.
%We recall that $\sigma_\alpha$ is an increasing enumeration of $L_\alpha$.
\begin{lemma}
   \label{lemma:jolt}
   Let $\theta\in[0,2\pi)$,
   $M=M(\mathscr{L},\theta)$, $\alpha\in\Lambda_2$ and $\beta_n\in\Lambda$
   be an increasing sequence with limit $\alpha$.
   Then there is a sequence $\xi_n$, also converging to $\alpha$, with $\xi_n\in L_{\beta_n}$ and 
   such that 
   $$ \jolt_{\alpha,\beta_n}(\xi_n) = \theta.$$  
\end{lemma}
\begin{proof}
   By construction, $a_\alpha$ is equal to $\theta\cdot z_\alpha\text{ (mod }2\pi)$ on ordinals.
   Take $\xi_n$ given by Corollary \ref{cor:z_alpha}
   such that $z_\alpha(\xi_n) = z_{\beta_n}(\xi_n) + z_\alpha(\beta_n) + 1$.
   Hence, $a_\alpha(\xi_n) = a_{\beta_n}(\xi_n) + a_\alpha(\beta_n) + \theta\text{ (mod }2\pi)$.
   Isolating $\theta$, we obtain $\jolt_{\alpha,\beta_n}(\xi_n) = \theta$.
\end{proof}

We are now almost ready to prove Nyikos' theorem. But
as a warm-up, let us prove the following, which is simpler than
Theorem \ref{thm:main} below because it does not use Theorem \ref{thm:Nyikos} in its proof.
Notation reminder: $A'$ is the set of limit points of $A$.
\begin{thm}
    \label{thm:omega_1}
    If $\mathscr{L}$ is a $\clubsuit_C$-sequence
    and $\theta\in(0,2\pi)$, then 
    $M=M(\mathscr{L},\theta)$ does not contain a copy of $\omega_1$.
\end{thm}
\begin{proof}
   Suppose that there is one which we call $W$, then it must be unbounded in $M$ and
   we may assume that it is located in $\pi^{-1}(C)$ for a club $C$
   and that each fiber contains exactly one point $x_\alpha=\langle\alpha,\phi_\alpha\rangle\in M_\alpha\cap W$, 
   $\alpha\in C$.
   Let $S$ be stationary such that $L_\alpha\subset C$ when $\alpha\in S$.
   By Lemma \ref{lemma:converge}, for each $\alpha\in C$ there is $\beta(\alpha)$
   such that 
   $$
       |a_\alpha(\eta) +\phi_\alpha - \phi_\eta|< \frac{\theta}{4}
   $$
   for each $\eta$ between $\beta(\alpha)$ and $\alpha$.
   By Fodor there is some $\beta$ and a stationary
   $S_0\subset S\cap C$ such that $\beta(\alpha) = \beta$ for all $\beta\in S_0$.
   Take a point $\alpha>\beta$ in $S_0'$, then there is a sequence $\eta_n\in S_0$ with supremum $\alpha$.
   We may assume (by using the action of $\SSS^1$) that $\phi_\alpha = 0$. 
   By Lemma \ref{lemma:jolt}, there is a sequence $\xi_n\in L_{\eta_n}$ converging to $\alpha$
   with $\jolt_{\alpha,\eta_n}(\xi_n) = \theta$.
   Fix some $n$ such that $\beta<\xi_n$, and let $\eta=\eta_n$, $\xi=\xi_n$.
   Since $S_0\subset S$, $L_{\eta}\subset C$, hence $\langle \xi,\phi_\xi\rangle\in W$.
   By closedness, since $L_{\eta}\subset C$, $\langle\eta,\phi_\eta\rangle\in W$ as well.
   Since $\eta\in S_0$, $\phi_\xi$ is at distance at most $\frac{\theta}{4}$ from $a_{\eta}(\xi)+\phi_\eta$.
   Since $\alpha\in S_0$, $\phi_\xi$ is at distance at most $\frac{\theta}{4}$ from $a_{\alpha}(\xi)$
   and $\phi_\eta$ at distance at most $\frac{\theta}{4}$ from $a_{\alpha}(\eta)$.
   Hence:
   \begin{align*}
      \theta & = \jolt_{\alpha,\eta}(\xi) \\
        & =  a_\alpha(\xi) - a_\alpha(\eta) - a_\eta(\xi)   \\
        & = a_\alpha(\xi) -\phi_\xi +\phi_\xi -a_\alpha(\eta)-\phi_\eta + \phi_\eta - a_\eta(\xi)  \\
        & \le |a_\alpha(\xi)-\phi_\xi | + |\phi_\xi - (a_\eta(\xi)+\phi_\eta)| + |\phi_\eta - a_\alpha(\eta)|  \\
        & < \frac{\theta}{4} + \frac{\theta}{4} + \frac{\theta}{4} < \frac{3\theta}{4}
   \end{align*}
   a contradiction. 
\end{proof}

Now, the complete theorem.
\begin{thm}
   \label{thm:main}
   Let $\mathscr{L}=\langle L_\alpha\,:\,\alpha\in\Lambda\rangle$ be a $\clubsuit_C$-sequence.
   Then the following holds.
   \\
   (a) $M(\mathscr{L},1)$ is of order $\infty$.\\
   (b) $M(\mathscr{L},\frac{2\pi}{n})$ is of order $n$ for each $n\ge 1$.
\end{thm}
\begin{proof} \ \\
   (a) Let $M= M(\mathscr{L},1)$ and
   $F\subset M$ be club.
   If (a) of Theorem \ref{thm:Nyikos} does not hold, there is some club 
   $E\subset F$ satisfying (b). We show that it leads to a contradiction.
   The points in the fibers $M_\alpha\cap E$
   are at distance $2\pi/n$ from the closest other point.
   Let $\delta=\min\{|\frac{2\pi\cdot k}{n} - 1|\,:\,k=0,\dots,n-1\}$, then $\delta>0$.
   Let $S$ be stationary such that $L_\alpha\subset C$ when $\alpha\in S$.
   Let $\alpha\in C$,
   and $\{\langle\alpha,\phi+k\cdot 2\pi/n\rangle\,:\,k=1,\dots,n-1\}$ be the members
   of $E\cap M_\alpha$. Fix some $\epsilon$ with $0<\epsilon <\frac{\delta}{4}$.
   By Lemma \ref{lemma:Fclosed}, if $\alpha\in C$ there is $\beta(\alpha)$
   such that,
   when $\langle\eta,\phi\rangle\in M_\eta\cap E$ and $\beta(\alpha)<\eta<\alpha$,
   then there is a {\em unique} $\theta(\eta,\varphi)$ 
   such that $\langle\alpha,\phi(\eta,\varphi)\rangle\in M_\alpha\cap E$ and:
   $$
      |a_\alpha(\eta) +\theta(\eta,\varphi) -\phi |< \epsilon.
   $$
   The proof is then very similar to that of Theorem \ref{thm:omega_1}, but 
   we phrase it a bit less formally. 
   Take $\alpha\in S_0'$, 
   where $S_0\subset S\cap C$ 
   is stationary with $\beta(\alpha)=\beta$ for each $\alpha\in S$, 
   and $\eta_m$ a sequence converging to $\alpha$ in $S_0$.
   Let $U_\alpha^k$ and $U_{\eta_m}^k$ ($k=1,\dots,n$) be the strips of width $\epsilon$ around the points
   of $E\cap M_\alpha$ and $E\cap M_{\eta_m}$,
   which respectively follow $a_\alpha$ and $a_{\eta_m}$
   and go all the way down to $\beta$.
   There are $\theta_\alpha$ and $\theta_{\eta_m}$ such that
   these strips are respectively centered around $a_\alpha + \theta_\alpha + 2k\pi/n$
   and $a_{\eta_m} + \theta_{\eta_m} + 2k\pi/n$.
   By Lemma \ref{lemma:jolt} we may choose $m$ such that for $\eta = \eta_m$ and 
   some $\gamma\in L_\eta$, $\gamma>\beta$,
   we have $\jolt_{\alpha,\eta}(\gamma) = 1$.
   Since $\epsilon <\delta/4$, the strips $\wt{U}_\alpha^k$ 
   of width $2\epsilon$ around $a_\alpha + \theta_\alpha + 2k\pi/n - 1$
   are disjoint from each $U_\alpha^k$.
   Take $\langle\eta,\theta_\eta\rangle\in S^\eta_0\subset M_\eta\cap E$, then 
   it is in $U_\alpha^k$ for
   some $k$. Thus the distance between $\theta_\eta$ and $a_\alpha(\eta) + \theta_\alpha + 2k\pi/n$
   is less than $\epsilon$.
   Since $\jolt_{\alpha,\eta}(\gamma) = 1$, it means that the strip around
   $a_\eta + \theta_\eta$ is contained in $\wt{U}_\alpha^k$ when intersected with $M_\gamma$,
   and hence disjoint from all the $U_\alpha^\ell$, see Figure \ref{fig:thm}.
   This is a contradiction, because
   all the points in $M_\gamma\cap E$ are in the intersection of one of $U_\alpha^k$ with one of $U_\eta^\ell$.\\
   (b) 
   Let $M=M(\mathscr{L},\frac{2\pi}{n})$.
   Since the jolts given by Lemma \ref{lemma:jolt}
   are exactly of $\frac{2\pi}{n}$ and the ``transitions'' of the various $a_\alpha$
   happen only between consecutive ordinals, for each $\alpha,\beta,\gamma$ with $\gamma<\beta<\alpha$,
   the difference $a_\alpha(\gamma)-a_\beta(\gamma)$ 
   is $k\cdot\frac{2\pi}{n}$
   for some $k\in\Z$. It follows that the set $\omega_1\times\{\frac{2k\pi}{n}\,:\,k=0,\dots,n-1\}$
   is club in $M$, showing that $M$ is of order at most $n$.
   Suppose that it is of order $m<n$ and as in (a) take $F\subset M$ be club
   such that $\pi(F)=C$ and $F$ satisfies (b) of Theorem \ref{thm:Nyikos} with $m$,
   in particular points of $F$ in the same fiber have $\frac{2\pi}{m}$-gaps in between.
   Since $m<n$, $\delta = |\frac{2\pi}{m}-\frac{2\pi}{n}| > 0$.
   The proof can now proceed exactly as in (a), with some $\epsilon<\delta/4$, until a contradiction is
   reached.
\end{proof}

\begin{figure}
   \begin{center}
       \epsfig{figure=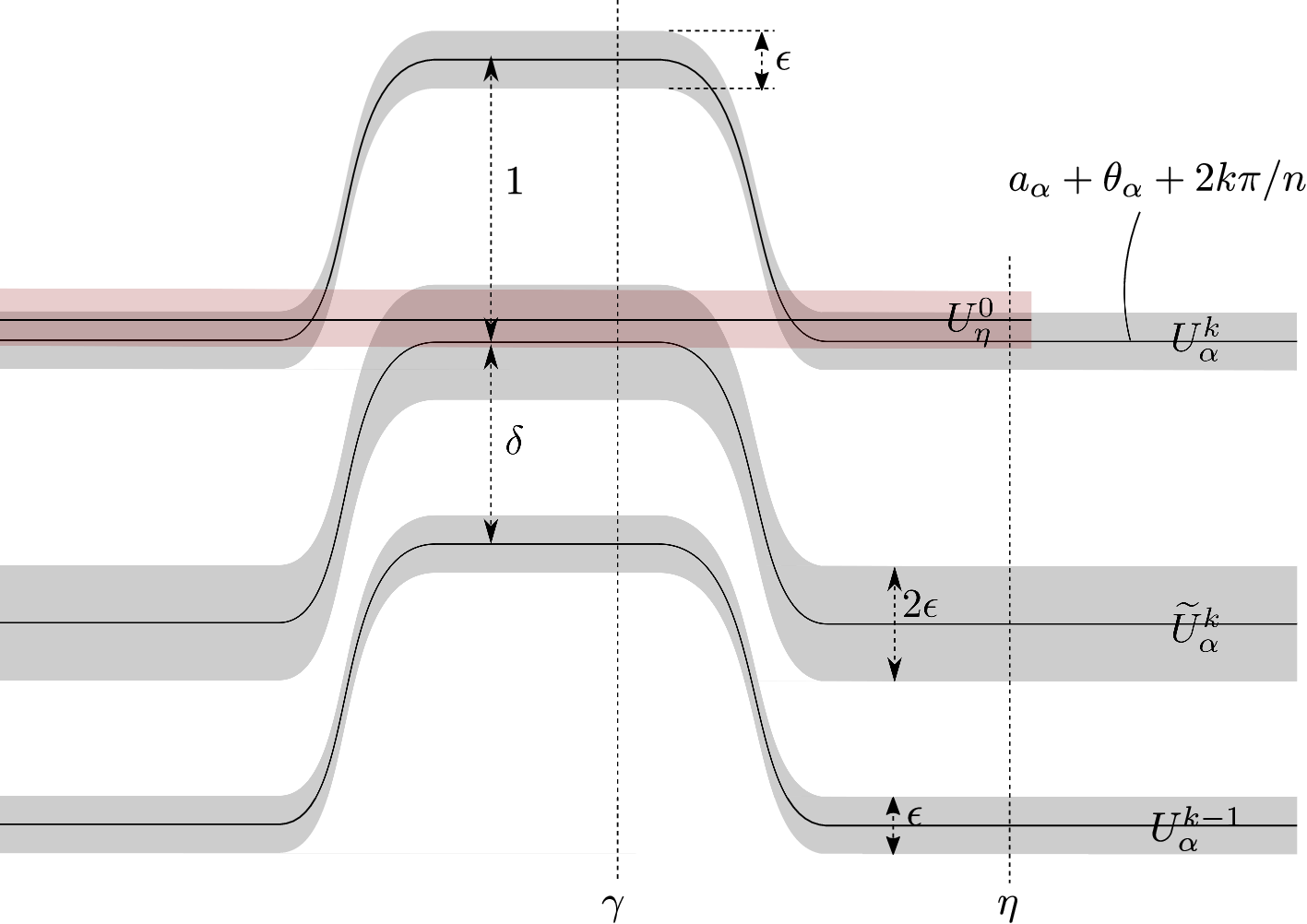, width=.75\textwidth}
       \caption{The proof of Theorem \ref{thm:main}.}
       \label{fig:thm}
   \end{center}
\end{figure}

An immediate consequence is:
\begin{proof}[Proof of Theorem \ref{thm:intro}]
   Take a principal $\SSS^1$-bundle over $\LL_+$ of order $\infty$
   and collapse $\pi^{-1}((0,1])$ to a point, or mirror the construction to build a bundle
   over the long line $\LL$, made of two copies of $\LL_+$ glued at a common $0$-point. 
   This yields a countably compact non-compact hereditarily collectionwise normal manifold without any copy of $\omega_1$
   by Lemmas \ref{lemma:copyomega_1} and \ref{lemma:hcwn}.
\end{proof}

%%%%%%%%%%%%%%%%%%%%%%%%%%%%%%%%%%%%%%%%%%%%%%%%%%%%%%%%%%%%%%%%%%%%%%%%%%%%%%%%%%%%%%%%%%%%%%
%%%%%%%%%%%%%%%%%%%%%%%%%%%%%%%%%%%%%%%%%%%%%%%%%%%%%%%%%%%%%%%%%%%%%%%%%%%%%%%%%%%%%%%%%%%%%%
%%%%%%%%%%%%%%%%%%%%%%%%%%%%%%%%%%%%%%%%%%%%%%%%%%%%%%%%%%%%%%%%%%%%%%%%%%%%%%%%%%%%%%%%%%%%%%

\section{Finite-to-$1$ closed preimages of $\omega_1$}
\label{sec:finitetoone}

This section started up by merely writing a few remarks linking $\SSS^1$-bundles over $\LL_+$
to another type of spaces that Nyikos has investigated, which can be seen as 
a discrete version of them. We then decided to add a small new result to make things more interesting,
which pushed us in the territory of elementary group theory. The said territory is a kind
of ``hic sunt leones'' place for us, the result being a bloated number of elementary lemmas
that we had to write down explicitely in order to convince ourselves that 
we were not about to be eaten by a lion lurking in these lands. 
We hope that the end product remains somewhat readable.
In the definition below, $n$ is identified with the set $\{0,1,\dots,n-1\}$ (i.e. we view $n$
as an ordinal).

\begin{defi}
   Let $n\in\omega$. A closed $n$-to-$1$ preimage of $\omega_1$ is the set $\omega_1\times n$ endowed
   with a Hausdorff topology such that the projection on the first factor $\pi$ is a closed continuous
   map. We abbreviate ``closed $n$-to-$1$ preimage of $\omega_1$'' by ``$n$-cp$\omega_1$''.
\end{defi}

%%%%%%%%%%%%%%%%%%%%%%%%%%%%%%%%%%%%%%%%%%%%%%%%%%%%%%%%%%%%%%%%%%%%%%%%%%%%%%%%%%
\subsection{Generalities on $n$-cp$\omega_1$'s}

If $K$ is a $n$-cp$\omega_1$ and $\alpha\in\omega_1$, then there are disjoint open $U_0,\dots,U_{n-1}$
containing each $\langle \alpha,i\rangle$. If $\alpha$ is limit, there must be some
$\beta<\alpha$ such that $\cup_i U_i\supset \pi^{-1}((\beta,\alpha])$.
Indeed, otherwise there is a sequence $\langle \alpha_n,i_n\rangle$ outside of this union
with $\alpha_n\to\alpha$, but the image of this sequence by $\pi$ is not closed.
We may thus proceed by induction and define $B(\alpha,i)$ to be clopen such that $B(\alpha,0),\dots,B(\alpha,n-1)$
partition $\pi^{-1}([0,\alpha])$, and $B(\alpha,i)\ni\langle\alpha,i\rangle$. 
These $B(\alpha,i)$ intersected with $\pi^{-1}((\beta,\alpha])$ for $\beta<\alpha$ form a neighborhood
basis of $\langle\alpha,i\rangle$.
(This is what we mean later on when we say that the collection of the $B(\alpha,i)$ engenders the topology.)
Of course, the initial parts of the $B(\alpha,i)$ do not matter for the topology (as long as they are open).
In particular, for successor ordinals
we may take $B(\alpha+1,i)=\{\langle\alpha+1,i\rangle\}\cup B(\alpha,i)$.
Going the other way, one may define the topology on $\omega_1\times n$ 
by specifying a collection $\mathscr{B} = \{B(\alpha,i)\,:\,\alpha\in\omega_1,\,i\in n\}$,
with the $B(\alpha,i)\subset[0,\alpha]\times n$, $B(\alpha,i)\ni\langle\alpha,i\rangle$ and
such that the $B(\alpha,i)$ partition $[0,\alpha+1]\times n$
(the intervals are taken in $\omega_1$).
We then say that $\mathscr{B}$ is a {\em family in $\omega_1\times n$}.
The resulting topological space is denoted by $K(\mathscr{B})$.
From now on, cursive letters $\mathscr{A},\mathscr{B},\dots$ always denote such families,
with the same letter (in roman capitals) used for the subsets, i.e. $A(\alpha,i),B(\alpha,i),\dots$

\begin{lemma}
   A family $\mathscr{B}$ engenders a topology of an $n$-cp$\omega_1$ iff
   $\forall\beta<\alpha$ there is $\gamma<\beta$ such that 
   \begin{equation}
      \label{eq:engenders}
      B(\alpha,i)\cap\bigl((\gamma,\beta]\times n\bigr) = 
      \bigcup_{j\,:\,\langle\beta,j\rangle\in B(\alpha,i)}
      B(\beta,j)\cap\bigl((\gamma,\beta]\times n\bigr). 
   \end{equation}
\end{lemma}
\begin{proof}
   Immediate.
\end{proof}

Of course, two distinct families may yield homeomorphic $n$-cp$\omega_1$'s.
We now define two classes of families. 

\begin{defi}
   In $\omega_1\times n$,
   $\mathscr{B}$
   is transverse if $B(\alpha,i)\cap\pi^{-1}(\{\beta\})$
   is a singleton for each $\beta\le\alpha$, and banded if 
   for all $\alpha\in\Lambda$ and $\beta<\alpha$,
   $B(\alpha,i)$ contains $\{\beta\}\times n$
   whenever $B(\alpha,i)\cap\{\beta\}\times n\not=\varnothing$.
   A topology on $\omega_1\times n$ is transverse/banded iff
   there is some collection $\mathscr{B}$ which engenders the topology
   such that the resulting space $K(\mathscr{B})$ is transverse/banded.
\end{defi} 

If $n=2$, transverse families are called symmetric by Nyikos, for obvious reasons. 
We felt that this new terminology fits better with the situation when $n\ge 3$.
We prefer to fix the underlying set and to talk about properties of topologies 
instead of that of the space 
up to homeomorphism because of 
the following.

\begin{example}[{\bf Nyikos \cite{Nyikoscoherent}}]
   \label{ex:banded}
   For each $n$ there is
   a banded family $\mathscr{B}$ in $\omega_1\times n$ such 
   that $K(\mathscr{B})$ is homeomorphic to $\omega_1$. If $n\ge 3$ there is
   a family $\mathscr{C}$ which is neither banded nor transverse
   such that $K(\mathscr{C})$ is homeomorphic to $\omega_1\times 2$ with the product topology.
\end{example}
\begin{proof}[Details]
   Set $B(\alpha,i) = \{\langle\alpha,i\rangle\}$ for each $i\ge 1$ and $\alpha\in\omega_1$.
   Hence, $B(\alpha,0) = \{\langle\alpha,0\rangle\}\cup\left([0,\alpha)\times n\right)$.
   This is easily seen to be homeomorphic to $\omega_1$.
   For $\mathscr{C}$, set $C(\alpha,i)=B(\alpha,i)$ when $i=0,1,\dots,n-1$ and
   $C(\alpha,n) = [0,\alpha]\times\{n\}$.
\end{proof}

\begin{lemma}
   \label{lemma:symbundle}
   If $\mathscr{B}$ is a transverse then in $K(\mathscr{B})$, $\pi^{-1}([0,\alpha])$ is homeomorphic to
   $[0,\alpha]\times n$ with the product topology for each $\alpha$.
\end{lemma}
\begin{proof}
   The $B(\alpha,i)$
   are disjoint clopen sets containing a point in each fiber.
\end{proof}

\begin{cor}
   If $\mathscr{A},\mathscr{B}$ are transverse, satisfy (\ref{eq:engenders}), and $K(\mathscr{A})$ and 
   $K(\mathscr{B})$ are homeomorphic, then they are both $n$-cp$\omega_1$ (for the same $n$).
\end{cor}
\begin{proof}
   Any homeomorphism sends $\pi^{-1}(\{\alpha\})$ of the first one to $\pi^{-1}(\{\alpha\})$
   of the second one for a club set of $\alpha$, and we conclude with the previous lemma.
\end{proof}

The previous corollary seems a triviality but does not hold for non-transverse $n$-cp$\omega_1$,
as seen in Example \ref{ex:banded}.
Notice also the following.

\begin{cor}
   \label{cor:covering}
   Let $\mathscr{B}$ be a transverse family in $\omega_1\times n$.
   Then $\pi:K(\mathscr{B})\to\omega_1$ is a $n$-sheeted covering map.
\end{cor}
Recall that a continuous $f:X\to Y$ is a {\em covering map} iff for each $y\in Y$ there is an open $U_y\ni y$
such that $f^{-1}(U_y)$ is a discrete union of subspaces of $X$ each homeomorphic to $U_y$.
The covering is $n$-sheeted iff $f^{-1}(U_y)$ is a discrete union of exactly $n$ copies of $U_y$ for each $y\in Y$.
\begin{proof}
   It is immediate that $f^{-1}(B(\alpha,i))$ is a discrete union of $n$ copies of it by Lemma \ref{lemma:symbundle}.
\end{proof}

If $\mathscr{B}$ is transverse,
then an equivalent way to define $B(\alpha,i)$ is 
by giving a function $\nu_\alpha:[0,\alpha]\to \mathfrak{S}_n$,
where $\mathfrak{S}_n$ is the group of permutations of $\{0,\dots,n-1\}$,
letting  $\nu_\alpha(\beta)(i) = j$ for the unique $j\in n$ such that $\langle\gamma,j\rangle\in B(\alpha,i)$.
Notice that by definition $\nu_\alpha(\alpha) = id$.
Conversely, given such a collection of $\nu_\alpha:\alpha\to \mathfrak{S}_n$
(with $\nu_\alpha(\alpha) = id$),
set $B(\alpha,i) = \{\langle\beta,\nu_\alpha(\beta)(i)\rangle,\,\beta\le\alpha\}$.
We denote by $\Upsilon(\mathscr{B})$ the collection of $\nu_\alpha$ obtained from the family
$\mathscr{B}$ and by $\mathscr{B}(\Upsilon)$ the family obtained from $\Upsilon$.
By definition, $\mathscr{B}(\Upsilon(\mathscr{B})) = \mathscr{B}$ and $\Upsilon(\mathscr{B}(\Upsilon)) = \Upsilon$.
Let us write once for all what we need to define a transverse $n$-cp$\omega_1$ from this data.

\begin{lemma}
   \label{lemma:transversestuff}
   Let $\mathscr{B}=\{B(\alpha,i)\,:\,\alpha\in\omega_1,\,i\in n\}$ be a family in $\omega_1\times n$ 
   %satisfying (\ref{eq:engenders})
   such that $|B(\alpha,i)\cap\{\beta\}\times n| = 1$ for all $\beta\le\alpha$.
   Let $\Upsilon(\mathscr{B})$ be defined as above.
   Then the topology engendered by the $B(\alpha,i)$ is transverse
   iff 
   \begin{equation}
   \label{eq:Balpha}
      \forall \beta<\alpha\,\exists\gamma<\beta\,
      \text{ such that if }\langle\beta,j\rangle\in B(\alpha,i)\text{ then }
       B(\alpha,i)\cap\pi^{-1}((\gamma,\beta]) = B(\beta,j)\cap\pi^{-1}((\gamma,\beta])
   \end{equation}
   iff
   \begin{equation}
   \label{eq:nualpha}
      \forall \beta<\alpha\,\exists\gamma<\beta \text{ such that } 
      \nu_\alpha(\eta) = \nu_\beta(\eta)\circ\nu_\alpha(\beta)
      \,\forall \gamma<\eta\le\beta.
   \end{equation}
\end{lemma}
\begin{proof}
   Both conditions are equivalent by definition, 
   and are equivalent to (\ref{eq:engenders}) since $|B(\alpha,i)\cap\{\beta\}\times n| = 1$.
   Hence, $\mathscr{B}$ engenders a topology of an $n$-cp$\omega_1$.
\end{proof}

We now fix the dimension $n$.
Given two collections $\mathscr{A},\mathscr{B}$ 
in $\omega_1\times n$,
we write $\mathscr{A} \sim \mathscr{B}$ iff $id:K(\mathscr{A})\to K(\mathscr{B})$ is an homeomorphism.
In other words: they define the same topology on the same underlying set. 

\begin{lemma}
   \label{lemma:symsim}
   Let $\mathscr{B}$ satisfy (\ref{eq:engenders}). 
   Suppose that for each $\alpha$, there is $\beta(\alpha)<\alpha$
   such that $B(\alpha,i)\cap\pi^{-1}(\{\xi\})$ 
   is a singleton for each $\xi$ with $\beta(\alpha)<\xi\le\alpha$.
   Then there is a transverse $\mathscr{A}$ such that $\mathscr{A} \sim \mathscr{B}$.
\end{lemma}
Note that $\mathscr{B}$ satisfies (\ref{eq:Balpha}).
\begin{proof}
   By induction on $\alpha$. If $\alpha=\beta+1$, take
   $A(\alpha,i)= \{\langle\alpha,i\rangle\}\cup B(\beta,i)$.
   If $\alpha$ is limit, take 
   $$ A(\alpha,i) = A(\beta(\alpha),j)\cup \left( B(\alpha,i) \cap \pi^{-1}((\beta(\alpha)+1,\alpha])\right) $$
   for the unique $j$ such that $B(\alpha,i)\ni\langle\beta,j\rangle$.
\end{proof}

Also, transverse families have a property very similar to that of $\SSS^1$-bundles of order $n$ over $\LL_+$.
We again say that a subspace is unbounded if its projection on $\omega_1$ is unbounded.

\begin{lemma}
   \label{lemma:symorder}
   Let $\mathscr{B}$ be transverse and satisfy (\ref{eq:Balpha}). For any club $E\subset K(\mathscr{B})$
   there is $m\le n$ and a club $C\subset\omega_1$ such that
   $|E\cap\{\alpha\}\times n| = m$ for each $\alpha\in C$. 
\end{lemma}
\begin{proof}
   Let $m$ be maximal such that $S_m = \{\alpha\in\omega_1\,:\,|E\cap\{\alpha\}\times n| = m\}$ is stationary.
   If $S_m$ contains a club, we are over.
   Otherwise, $\omega_1-S_m$ is stationary, and there is thus some $k<m$ such that $S_k$ is also stationary.
   By (\ref{eq:Balpha}) and Fodor, there is $\beta$ and a stationary $\wt{S}\subset S_k$ such that
   $$
      \bigcup_{i:\langle\alpha,i\rangle\in E} B(\alpha,i) \cap (\beta,\alpha] = E\cap (\beta,\alpha]
   $$
   when $\alpha\in \wt{S}$.
   Taking $\eta\in S_m$ and $\alpha\in \wt{S}$ with $\beta<\eta<\alpha$,
   we have that the union of $k$-many $B(\alpha,i)$ contain $m$ distinct points in $E\cap\{\eta\}\times n$,
   a contradiction since $k<m$ and $\mathscr{B}$ is transverse.
   Hence,  $\omega_1-S_m$ is not stationary and $S_m$ contains a club.
\end{proof}

\begin{defi}
   The order of $K$ is the smallest $m$ for which there is a club $E\subset K(\mathscr{B})$
   such that $|E\cap\{\alpha\}\times n| = m$ for each $\alpha\in\pi(E)$.
\end{defi}

%%%%%%%%%%%%%%%%%%%%%%%%%%%%%%%%%%%%%%%%%%%%%%%%%%%%%%%%%%%%%%%%%%%%%%%%%%%%%%%%%%%%%%%%%%%%%%%%%%%%%%%%%%%
\subsection{Principal $\Z_n$-bundles over $\omega_1$}

To avoid annoying examples such as \ref{ex:banded},
we may restrict the class of homeomorphisms between $n$-cp$\omega_1$.

\begin{defi}
   Let $K_0,K_1$ be $n$-cp$\omega_1$. Then an homeomorphism
   $h:K_0\to K_1$ is a bundle-homeomorphism if
   it leaves the fibers $\pi^{-1}(\{\alpha\})$ invariant.
\end{defi} 
Here, we denote both projections of $K_0,K_1$ by $\pi$ since they share the same underlying set
$\omega_1\times n$ and $\pi$ depends only on it.
The term ``bundle'' used in a context where the base space is not connected is slightly unusual,
but is somewhat justified by Lemma \ref{lemma:symbundle} and the definition below.
The nice thing with this notion is the following.

\begin{lemma}
   Let $\mathscr{A},\mathscr{B}$ be families in $\omega_1\times n$ such that
   $K(\mathscr{A})$ and $K(\mathscr{B})$ are bundle-homeomorphic.
   Then $\mathscr{A}$ is banded [transverse] iff $\mathscr{B}$ is banded [transverse].
\end{lemma}
\begin{proof}
   A bundle homeomorphism is just a permutation when restricted to a fiber.
\end{proof}

Let us now emulate principal $\SSS^1$-bundles.
We write $\Z_n$ for the group of integers with addition modulo $n$.

\begin{defi}
   An $n$-cp$\omega_1$ $K$ is a principal $\Z_n$-bundle over $\omega_1$
   if there is a continuous free transitive group action of $\Z_n$ leaving each
   fiber $\pi^{-1}(\{\alpha\})$ invariant.
   We write the action on the fiber $\pi^{-1}(\{\alpha\})$ as
   $k,\langle\alpha,i\rangle\mapsto \langle\alpha,k{\bm{\cdot}_\alpha}i\rangle$.
\end{defi}

Any transverse $2$-cp$\omega_1$ is such a bundle:
the action is given by permuting $0$ and $1$. This sends $B(\alpha,0)$ to $B(\alpha,1)$
for any $\alpha$.
The situation is a bit more complex if $n\ge 3$. Actually,
whether an $n$-cp$\omega_1$ $K$ can be homeomorphic to a principal $\Z_n$-bundle over $\omega_1$ 
is not very clear when $n$ is not fixed.

\begin{example}
   A family $\mathscr{B}$ in $\omega_1\times 2$ and 
   a family $\mathscr{C}$ in $\omega_1\times 3$ such that $K(\mathscr{B})$ and $K(\mathscr{C})$
   are homeomorphic, $K(\mathscr{B})$ is a principal $\Z_2$-bundle but 
   $K(\mathscr{C})$ is not a principal $\Z_3$-bundle over $\omega_1$.
\end{example}
\begin{proof}[Details]
   This is a special case of Example \ref{ex:banded}.
   Set $B(\alpha,i)=[0,\alpha]\times\{i\}$,
   $C(\alpha,0)=[0,\alpha]\times\{0\}$,
   $C(\alpha,1) = \{\langle\alpha,1\rangle\}\cup[0,\alpha)\times\{1,2\}$
   and $C(\alpha,2) = \{\langle\alpha,2\rangle\}$.
   One sees easily (i.e. by applying Lemma \ref{lemma:bundlesym} below) that
   $\mathscr{C}$ does not yield a principal $\Z_3$-bundle.
\end{proof}

This annoyance disappears when one looks only at bundle-homeomorphisms,
as the action can be carried by it, hence two $n$-cp$\omega_1$ $K$ which are 
bundle-homeomorphic either both possess such an action, or both do not.

\begin{lemma}
   Let $K_0$ be a principal $\Z_n$-bundle over $\omega_1$ and $h:K_0\to K_1$ be a bundle-homeomorphism.
   Then $h$ induces a permutation $\sigma_\alpha$ on $\{\alpha\}\times n$,
   and $K_1$ is a principal $\Z_n$-bundle over $\omega_1$ with action 
   $\star_\alpha$ given by
   $k,\langle\alpha,i\rangle\mapsto
    \langle\alpha,k\star_\alpha i\rangle=\langle\alpha,\sigma_\alpha(k\bm{\cdot}_\alpha\sigma^{-1}_\alpha(i))\rangle$.
\end{lemma}
\begin{proof}\ \\
   \begin{center} 
     \begin{tikzcd}
     \sigma_\alpha(i) 
       \arrow[r, rightarrow , "\star_\alpha"]
       \arrow[d, leftarrow ,"\sigma_\alpha"]
     &
     k \star_\alpha i
       \arrow[d, leftarrow , "\sigma_\alpha"]
     \\
     i
       \arrow[r, rightarrow ,"\bm{\cdot}_\alpha"]
     &
     k\bm{\cdot}_\alpha i 
     \end{tikzcd}
   \end{center}
\end{proof}

\begin{rem}
   There is an unfortunate clash of notation between the permutation $\sigma_\alpha$ and the 
   enumeration $\sigma_\alpha(n)$ of $L_\alpha$ in a ladder system used in previous sections.
   Since no ladder system is used in this subsection, this should not cause too much confusion.
\end{rem}

\begin{lemma}
   \label{lemma:bundlesym}
   If a topology $\tau$ on $\omega_1\times n$ yields a
   principal $\Z_n$-bundle $K$ over $\omega_1$, then $\tau$ is transverse. 
   Moreover, the order of $K$ divides $n$.
\end{lemma}

In what follows, $k\bm{\cdot} E$ denotes $\{\langle\alpha,k\bm{\cdot}_\alpha i\rangle\,:\,\langle\alpha,i\rangle\in E\}$,
i.e. the image of $E$ under the action of $k$.
\begin{proof}
   Define the $B(\alpha,i)$ starting from $\tau$ as above.
   Since the continuous
   action leaves the fibers invariant and the $B(\alpha,i)$ partition $\{\alpha\}\times n$,
   for each $k,i$, there is some $\beta_i<\alpha$ and some $j$ such that 
   $\left(k{\bm{\cdot}}B(\alpha,i)\right)\cap[\beta_i,\alpha]\times n = B(\alpha,j)\cap[\beta_i,\alpha]\times n$.
   Let $\beta$ be $\max_i\beta_i$, then 
   since the action is transitive
   each $B(\alpha,i)$ intersects
   each fiber $\{\xi\}\times n$ for each $\beta<\xi\le\alpha$.
   They must thus intersect it just once and we may conclude with Lemma 
   \ref{lemma:symsim}.
   \\
   Now, let $m$ be the order of $K$.
   If $E$ is club in $K$ such that $|E\cap\{\alpha\}\times n|=m$ for each $\alpha\in\pi(E)$,
   there is a club $C$ such that
   for each $\alpha\in C$ we have either
   $k{\bm{\cdot}} E\cap \left(\{\alpha\}\times n \right)= E\cap \left(\{\alpha\}\times n\right)$
   or $\left(k{\bm{\cdot}} E\right)\cap E\cap \left(\{\alpha\}\times n \right)= \varnothing$. 
   Indeed, if not, then there is a stationary $S$ such that
   $\left((k{\bm{\cdot}} E)\cap E\right)\cap \left(\{\alpha\}\times n\right)$ and 
   $\left((k{\bm{\cdot}}E) - E\right)\cap \left(\{\alpha\}\times n\right)$
   are non-empy for $\alpha\in S$. In particular, the club set $F = \left(k{\bm{\cdot}} E\right)\cap E$ 
   satisfies $|F\cap\{\alpha\}\times n|<m$ for each $\alpha\in S$.
   Hence there is a stationary $S_0\subset S$ and $k<m$ such that $|F\cap\{\alpha\}\times n|=k$
   when $\alpha\in S_0$. This shows that the order of $K$ is $\le k<m$ by
   Lemma \ref{lemma:symorder}, a contradiction.
   Hence, the images of $E$ by the action of $\Z_n$ partition $n$ in pieces of size $m$, so $m$ divides $n$.
\end{proof}

The fact that each fiber is a discrete space yields complications when one 
tries to describle for which kind of families $\mathscr{B}$ in $\omega_1\times n$ is $K(\mathscr{B})$ 
a principal $\Z_n$ bundle. Call these {\em principal families}.
If $E\subset K$ and $i\in n$, we write $E^{\bm{+}}$ 
for $\{\langle\alpha,k+1\rangle\,:\,\langle\alpha,k\rangle\in E\}$.
It would be nice if principal families were simply those for which one obtains $B(\alpha,i+1)$ simply
by pushing up $B(\alpha,i)$: $B(\alpha,i+1)= B(\alpha,i)^{\bm{+}}$, for each $i$. 
(Of course, all additions are modulo $n$.)
But Lemma \ref{lemma:symsim} indicates that it cannot be {\em that} simple.
Also, nothing garantees that the action of $\Z_n$ is given by 
$\langle\alpha,k\bm{\cdot}_\alpha i\rangle=\langle\alpha,i+k\rangle$.
One of the problems lies in the fact that 
we may permute the members of some (or each) $\pi^{-1}(\{\alpha\})$,
altering in the same time the $B(\alpha,i)$ and the group action, rendering things more difficult
to decipher. We thus take a moment to talk about these global permutations,
which forces us to switch to group theoretic considerations (which remain, thankfully for us, elementary).
\\
Recall that $\mathfrak{S}_n$ is the group of permutations on $n=\{0,1,\dots,n-1\}$.
A {\em global permutation} $\bm{\sigma}$ of $\omega_1\times n$ is a 
map 
$\omega_1\to\mathfrak{S}_n$. We write $\sigma_\alpha$ for $\bm{\sigma}(\alpha)$.
Given a family $\mathscr{B}$ in $\omega_1\times n$, we define $\bm{\sigma}\mathscr{B}$
as the family 
\begin{equation}
    \label{eq:sigmaB}
    \bm{\sigma}B(\alpha,i) = 
    \{\langle\beta,\sigma_\beta(j)\rangle\,:\,\beta\le\alpha,\,\langle\beta,j\rangle\in B(\alpha,i)\}.
\end{equation}
If $\tau$ is the topology on $\omega_1\times n$ engendered by $\mathscr{B}$,
the topology engendered by $\bm{\sigma}\mathscr{B}$ is simply
$\tau_{\bm{\sigma}}=\{\bm{\sigma}(U)\,:\,U\in\tau\}$.
The map 
$h_{\bm{\sigma}}:\langle K(\mathscr{B}),\tau\rangle \to \langle K(\bm{\sigma}\mathscr{B}),\tau_{\bm{\sigma}}\rangle$
defined as $\langle\alpha,i\rangle\mapsto\langle\alpha,\sigma_\alpha(i)\rangle$ is obviously a bundle-homeomorphism. 
In a sense, $\bm{\sigma}$ proceeds by simply relabelling the fibers.

\begin{lemma}
   Let $\mathscr{A},\mathscr{B}$ be transverse such that $K(\mathscr{A})$ and $K(\mathscr{B})$ are bundle-homeomorphic.
   Then there is a global permutation $\bm{\sigma}$ such that $\bm{\sigma}\mathscr{B}\sim\mathscr{A}$.
\end{lemma}
\begin{proof}
   A bundle-homeomorphism $h$ permutes the members of the fibers, hence $h=h_{\bm{\sigma}}$
   for some global permutation $\bm{\sigma}$. Thus $\bm{\sigma}\mathscr{B}$ and $\mathscr{A}$
   define the same topology on $\omega_1\times n$.
\end{proof}

\begin{lemma}
   Let $\bm{\sigma}$ be a global permutation and $\mathscr{B}$ be a transverse family in $\omega_1\times n$.
   Then $\Upsilon(\bm{\sigma}\mathscr{B})$, which we denote as $\bm{\sigma}\Upsilon$, 
   is given by the collection $\bm{\sigma}\nu_\alpha$, where 
   $$ \bm{\sigma}\nu_\alpha(\beta) = \sigma_\beta\circ\nu_\alpha(\beta)\circ\sigma_\alpha^{-1}.$$
\end{lemma}
\begin{proof}
   By (\ref{eq:sigmaB}), 
   $$ \sigma_\beta\Bigl(\nu_\alpha(\beta)\bigl(\sigma_\alpha(i)\bigr)\Bigr) = \sigma_\beta\bigl(\nu_\alpha(\beta)(i)\bigr).$$
   The result follows by applying $\sigma_\alpha^{-1}$ on the right.
\end{proof}

Contents-wise, the next result is just a glorified remark, but
it makes our life quite easier, hence its promotion as a proposition.

\begin{prop}
   \label{prop:lifeiseasy}
   If $K$ is a principal $\Z_n$-bundle over $\omega_1$, then there is 
   a global permuation $\bm{\sigma}$ of $\omega_1\times n$ such that after applying it,
   the action by $\Z_n$ becomes $\langle\alpha,k\bm{\cdot}_\alpha i\rangle = \langle\alpha,i+k\rangle$.
\end{prop}
\begin{proof}
   Since $\Z_n$ is a cyclic group of order $n$, any action of
   it on $\{0,1,\dots,n-1\}$
   can be seen as that of a cyclic subgroup of $\mathfrak{S}_n$ of order $n$. Let $\rho$ be a generator of this
   subgroup, and write $\mu$ for the map $i\mapsto i+1\text{ mod } n$ seen as a member of $\mathfrak{S}_n$.
   It is well known that two elements of $\mathfrak{S}_n$ are conjugate iff
   they have the same cycle type, hence there is $\sigma\in\mathfrak{S}_n$
   such that $\sigma\rho\sigma^{-1} = \mu$.
   We may thus define $\sigma_\alpha$ looking at the generator $\rho_\alpha$
   of the action of $\Z_n$ on $\{\alpha\}\times n$,
   and the situation is as in the diagram below.
   \begin{center} 
     \begin{tikzcd}
     \sigma_\alpha(i) 
       \arrow[r, rightarrow ,"\mu"]
       \arrow[d, leftarrow ,"\sigma_\alpha"]
     &
     \sigma_\alpha(i) +1\text{ (mod }n) 
       \arrow[d, leftarrow , "\sigma_\alpha"]
     \\
     i
       \arrow[r, rightarrow ,"\rho_\alpha"]
     &
     \rho_\alpha(i)\\
     \end{tikzcd}
   \end{center}
   Setting $\bm{\sigma}$ as $\bm{\sigma}(\alpha) = \sigma_\alpha$,
   we obtained the desired global permutation.
\end{proof}

\iffalse
\begin{rem}
   It is known that
   any two isomorphic subgroups of $\mathfrak{S}_n$ that act freely and transitively on $\{0,\dots,n-1\}$
   are conjugate, hence Proposition \ref{prop:lifeiseasy} holds more generally.
   If one removes this condition, there are counter-examples
   For instance, $\mathfrak{S}_4$ contains two non-conjugate subgroups both isomorphic to
   $\mathfrak{S}_2\times\mathfrak{S}_2$.
\end{rem}
\fi

\begin{cor}
   \label{cor:bundle}
   Let $\mathscr{A}$ be a transverse family of $\omega_1\times n$. 
   Then $K(\mathscr{A})$ is a principal $\Z_n$-bundle
   over $\omega_1$ iff
   there is a global permutation $\bm{\sigma}$
   such that, writing $\mathscr{B} = \bm{\sigma}\mathscr{A}$, for each $\alpha\in\omega_1$ there
   is $\beta(\alpha)<\alpha$ with
   \begin{equation} 
      \label{eq:bundlebeta}
      B(\alpha,i+1)\cap\pi^{-1}((\beta(\alpha),\alpha]) = B(\alpha,i)^{\bm{+}}\cap\pi^{-1}((\beta(\alpha),\alpha]).
   \end{equation}
\end{cor}
\begin{proof}
   First assume that $K(\mathscr{A})$ is a principal $\Z_n$-bundle
   over $\omega_1$.
   Take $\bm{\sigma}$ given by the proposition. We may thus assume that
   the action on each fiber is by addition mod $n$ in $K(\bm{\sigma}\mathscr{A})$.
   The rest of the proof is the same as that of
   \ref{lemma:bundlesym}: since $\langle\alpha,i\rangle\in B(\alpha,i)$,
   adding one means that $B(\alpha,i+1)$ is obtained by pushing up $B(\alpha,i)$,
   at least above some $\beta(\alpha)$.\\
   Conversely, up to applying $\bm{\sigma}$ we may assume that for each $\alpha$
   there is $\beta(\alpha)$ such that (\ref{eq:bundlebeta}) holds.
   But then the action $k,i\mapsto i+k\mod n$ on each fiber is continuous, free and transitive. 
\end{proof}

\begin{lemma}
   \label{lemma:bundlelast}
   If $\mathscr{B}$ is a family such that for each $\alpha$ there is $\beta(\alpha)<\alpha$
   such that (\ref{eq:bundlebeta}) holds, then there is a family $\mathscr{C}\sim\mathscr{B}$
   for which this holds with $\beta(\alpha) = 0$.
\end{lemma}
\begin{proof}
   Same proof as Lemma \ref{lemma:symsim}: build $C(\alpha,i)$ by induction.
\end{proof}

%%%%%%%%%%%%%%%%%%%%%%%%%%%%%%%%%%%%%%%%%%%%%%%%%%%%%%%%%%%%%%%%%%%%%%%%%%%%%%%%%%
\subsection{Principal $\Z_n$-bundles over $\omega_1$ of any order dividing $n$ under $\clubsuit_C$.}

In this short subsection, we show that it is very easy to use the functions $z_\alpha$ of Section \ref{sec:tool}
to obtain the spaces given in the title.

\begin{thm}[{$\clubsuit_C$}]
   \label{thm:Zmbundle}
   For any $k$ dividing $m$, there is a principal $\Z_m$-bundle over $\omega_1$ of order $k$.
\end{thm}
\begin{proof}
   Let $\ell$ be such that $\ell\cdot k = m$.
   Let $\mathscr{L}$ be a $\clubsuit_C$-sequence and $z_\alpha$ be modelled on $\clubsuit_C$
   (Definition \ref{defi:modelled}).
   We define the family $\mathscr{B}$ in $\omega_1\times m$ as
   $B(\alpha,i) = \{\ell\cdot z_\alpha(\beta) + i\text{ (mod }m)\,:\,\beta\le\alpha\}$.
   Then $\Z_m$ acts on $K(\mathscr{B})$ by addition in each fiber,
   so $K(\mathscr{B})$ is a principal $\Z_m$-bundle over $\omega_1$.
   Since $z_\alpha$ takes integer values, the subspaces
   $$ P_i = \{\langle \alpha,i + j\cdot \ell\rangle\,:\,\alpha\in\omega_1,\,j=0,\dots,k-1\} 
      =\cup_{0\le j\le k-1}B(\alpha,i+j\cdot\ell)
   $$
   are clopen in $K(\mathscr{B})$.
   This shows that the order of $K(\mathscr{B})$ is at most $k$.\\
   Each $P_i$ is bundle-homeomorphic to the $\Z_k$-bundle $K(\mathscr{A})$, with 
   $A(\alpha,i) = \{z_\alpha(\beta) + i\text{ (mod }k)\,:\,\beta\le\alpha\}$.
   We now show that $K(\mathscr{A})$ has order $k$.
   (The proof is very similar -- but easier -- than that for principal $\SSS^1$-bundles over $\LL_+$.)
   If its order is $\ell<k$, there is
   a club $E\subset K(\mathscr{A})$ which contains exactly $\ell$ points in each fiber that it intersects.
   Let $C=\pi(E)$, and $S$ be stationary such that $L_\alpha\subset C$ for $\alpha\in S$.
   By (\ref{eq:Balpha}), for each $\alpha\in C$, there is $\beta(\alpha)<\alpha$ such that
   $$ E \cap \pi^{-1}((\beta(\alpha),\alpha])\subset \cup_{\langle\alpha,i\rangle\in E} A(\alpha,i). $$
   Let $S_0\subset S\cap C$ be stationary with $\beta(\alpha) = \beta$ for each $\alpha\in S_0$.
   Take $\alpha\in(S\cap C)'$ and an increasing sequence $\eta_n\in S\cap C$ with limit $\alpha$.
   Take the $\xi_n\in L_{\eta_n}$ given by Corollary \ref{cor:z_alpha}.
   Fix $n$ such that $\xi_n>\beta$.
   By definition of $\mathscr{A}$, 
   $A(\alpha,i)$ contains $\langle \eta_n, z_\alpha(\eta_n) + i\rangle$.
   But if $A(\alpha,i)$ contains $\langle \xi_n, j\rangle$,
   then $A(\eta_n, z_\alpha(\eta_n) + i)$ contains $\langle \xi_n, j-1\rangle$ (by Corollary \ref{cor:z_alpha}).
   Hence, the $\ell$-many $A(\alpha,i)$ containing points of $E\cap\pi^{-1}(\{\alpha\})$
   and the $\ell$-many $A(\eta_n,i)$ containing points of $E\cap\pi^{-1}(\{\eta_n\})$
   are separated by $1$ in $\pi^{-1}(\{\xi_n\})$. This is not possible if $\ell<k$ since the $\ell$ points
   in $E\cap\pi^{-1}(\{\xi_n\})$ must belong to both.
\end{proof}

%%%%%%%%%%%%%%%%%%%%%%%%%%%%%%%%%%%%%%%%%%%%%%%%%%%%%%%%%%%%%%%%%%%%%%%%%%%%%%%%%%
\subsection{Relations between both types of principal bundles}

Another very short subsection containing only one result, which shows that the similarities between
the two types of principal bundles are genuine.

\begin{thm}  
   \label{thm:relationship}
   Let $M$ be a principal $\SSS^1$-bundle over $\LL_+$ of order $n$
   and $K$ be a principal $\Z_n$-bundle over $\omega_1$ of order $n$.
   Then the following holds.
   
   (a) $M$ contains a principal $\Z_n$-bundle of order $n$ over $\omega_1$.
   
   (b) If $F\subset M$ is a copy of 
       an $m$-cp$\omega_1$, then $m \ge n$ and
       there is a club $C\subset\omega_1$ and $E\subset F$ such that
       $E\cap\pi^{-1}(C)$ is a principal $\Z_n$-bundle of order $n$ over $\omega_1$.
   
   (c) There is a principal $\SSS^1$-bundle over $\LL_+$ of order $n$ that contains $K$.
\end{thm}
\begin{proof}
   Let $M$ be a principal $\SSS^1$-bundle over $\LL_+$ of order $n$.\\
   (a)
   Take the club $E$ given by Theorem \ref{thm:Nyikos} (b).
   Then the action of $\Z_n$ given by the rotation by $\frac{2\pi}{n}$ radians leaves $E$ invariant.
   By Lemma \ref{lemma:uncountable}, $E$ is of order $n$.
   \\
   (b)
   Take the $E\subset F$ given by Theorem \ref{thm:Nyikos} (b), and apply (a).
   \\
   (c)
   Let now $K$ be a principal $\Z_n$-bundle of order $n$ over $\omega_1$.
   By Corollary \ref{cor:bundle} and Lemma \ref{lemma:bundlelast},
   we may assume up to a global permutation that (\ref{eq:bundlebeta}) holds with $\beta(\alpha)=0$   
   for some family $\mathscr{B}$ which engenders the topology, and
   that the action of $\Z_n$ is by addition on each fiber.
   Let $b_\alpha(\gamma)\in n$ be 
   $\nu_\alpha(\beta)(0)$, that is:   
   the unique $i$ such that $\langle\gamma,i\rangle\in B(\alpha,0)$.
   Define $a_\alpha:(0,\alpha+1)\to\SSS^1$ (interval in $\LL_+$)
   as $a_\alpha(\gamma) = b_\alpha(\gamma)\cdot\frac{2\pi}{n}$ for each $\gamma<\alpha$ and interpolating between
   $a_\alpha(\gamma)$ and $a_\alpha(\gamma+1)$ in $[\gamma,\gamma+1]\subset\LL_+$.  
   By (\ref{eq:nualpha}), the topology defined by the family $a_\alpha$ is that of a
   principal $\SSS^1$-bundle which we call $M$.
   By construction, $K$ is contained in $M$ as the images of $b_\alpha(\gamma) + k\frac{2\pi}{n}$,
   hence the order of $M$ is at most $n$. 
   Write $F_K$ for this copy of $E$ in $M$.
   Let $E$ be club such that (b) of Theorem \ref{thm:Nyikos} holds for $m\le n$ minimal.
   By Lemma \ref{lemma:essinf}, 
   there is a club $C$ such that there are pairs of points respectively in $E$ and $K_F$
   whose separation is exactly $\theta\ge 0$ on the fibers above $C$.
   By the action of $\SSS^1$, there is a rotated copy $E_0$ of $E$ such that
   one point of $E_0$ and of $F_K$ coincide (in the fibers over $C$).
   If $m<n$, $G = E_0\cap F_K\cap\pi^{-1}(C)$ is a club subspace of $F_K$
   such that $M_\alpha\cap G$ contains $\le m<n$ points. 
   This contradicts Lemma \ref{lemma:symorder} since $F_K$ is a copy of $K$.
\end{proof}

We remark that the combination of Theorem \ref{thm:main} and \ref{thm:relationship} 
gives an indirect (much longer) proof of Theorem \ref{thm:Zmbundle}.

%%%%%%%%%%%%%%%%%%%%%%%%%%%%%%%%%%%%%%%%%%%%%%%%%%%%%%%%%%%%%%%%%%%%%%%%%%%%%%%%%%
\subsection{Transverse families that do not yield principal $\Z_n$-bundles}
\label{subsec:notprincipal}

Principal $\Z_n$-bundles over $\omega_1$ are transverse.
Since the fibers are discrete, one may get the impression that, 
by reordering them cleverly when needed, any transverse family yields a principal $\Z_n$-bundle.
This is actually not the case: some 
transverse families
are not principal.
The construction and proof is not difficult and is similar to what we did in Section \ref{sec:usingclub},
but it is somewhat tedious to write due to the presence of permutations in the fibers. 
We use the usual notation $(a_0, a_1, a_2,\dots, a_{k-1})$ for the permutation
sending $a_i$ to $a_{i+1\text{ (mod } k)}$.
From now on, $\mu$ denotes the permutation $(0, 1, 2, \dots, n-1)$ (that is: $i\mapsto i+1\text{ (mod } n)$
seen as a member of $\mathfrak{S}_n$).
Except when it disserves clarity, we use multiplicative notation for composition of members of $\mathfrak{S}_n$.

\begin{lemma}
   \label{lemma:conjmu}
   Let $\sigma\in\mathfrak{S}_n$.
   If $\sigma\mu\sigma^{-1} = \mu$, then $\sigma = \mu^k$ for some $k$.
\end{lemma}
\begin{proof}
   Conjugation in $\mathfrak{S}_n$ is relabelling, hence 
   $\sigma\mu\sigma^{-1} = (\sigma(0),\sigma(1),\dots,\sigma(n-1))$.
   It follows that $\sigma(i) = i+k\text{ (mod } n)$ for some $k$,
   and $\sigma = \mu^k$.
\end{proof}

We see $\Z_n$ as the subgroup of $\mathfrak{S}_n$ generated by $\mu$.
If $\mu$ acts on $\omega_1\times n$ as 
$\mu,\langle \alpha,i\rangle\mapsto \langle \alpha,\mu\bm{\cdot}_\alpha i\rangle$,
we write $\mu_\alpha$ for the permutation $i\mapsto \mu\bm{\cdot}_\alpha i$.
If $B\subset\omega_1\times n$, we write 
$\mu\bm{\cdot} B = \{\langle\alpha,\mu_\alpha(i)\rangle\,:\,\langle\alpha,i\rangle\in B\}$.

\begin{lemma}
   \label{lemma:actionbeta}
   If $K(\mathscr{B})$ is a principal $\Z_n$ bundle, then for each limit $\alpha$,
   there is $\beta(\alpha)<\alpha$ such that for each $i$,
   $$ \Bigl(\mu\bm{\cdot} B(\alpha,i)\Bigr) \cap \Bigl((\beta(\alpha),\alpha]\times n \Bigr)=  
      B(\alpha,\mu_\alpha(i)) \cap \Bigl((\beta(\alpha),\alpha]\times n\Bigr),
   $$
\end{lemma}
\begin{proof}
  The action by $\mu$ is an homeomorphism and thus must send a neighborhood
  of $\langle\alpha,i\rangle$ (i.e. $B(\alpha,i)$)
  to a neighborhood of $\langle\alpha,\mu_\alpha(i)\rangle$ (i.e. $B(\alpha,\mu_\alpha(i))$.
\end{proof}

Let us give a lemma that we may use in a concrete example.
\begin{lemma}
   \label{lemma:rho}
   Let $n\ge 3$ and
   $\mathscr{B}$ be family in $\omega_1\times n$.
   Suppose that there are $\gamma<\beta<\alpha$ such that
   $\bigl(\mu \bm{\cdot} B(\alpha,i)\bigr) \cap A = B(\alpha,\mu_\alpha(i))\cap A$
   and 
   $\bigl(\mu \bm{\cdot} B(\beta,i)\bigr) \cap A = B(\beta,\mu_\beta(i))\cap A$
   for $i=0,1,\dots,n-1$, with $A=\{\gamma,\beta,\alpha\}\times n$.
   Suppose moreover that there is some $\rho\in\mathfrak{S}_n$ such that the following holds for each $i$:
   
   (i) $B(\alpha,i)\ni\langle\beta,i\rangle$,
   
   (ii) $B(\beta,i)\ni\langle\gamma,i\rangle$,
   
   (iii) $B(\alpha,i)\ni\langle\gamma,\rho(i)\rangle$.\\
   Then, in $\mathfrak{S}_n$, $\rho$ is conjugate to $\mu^k$ for some $k$.
\end{lemma}

%The condition $\mu \bm{\cdot} B(\alpha,i) \cap A = B(\alpha,\mu_\alpha(i))\cap A$
%means that when one looks only at the fibers above $A$,
%the action by $\mu$ takes $B(\alpha,i)$ inside $B(\alpha,\mu_\alpha(i))$.
Notice that there are $\rho$ which are not conjugate to $\mu^k$ for any $k$ if $n\ge 3$:
$\mu^k$ is a product of cycles, each of the same length $\ell$ which must divide $n$.  
Thus, for instance, any cyclic permutation of length $n-1$ is not conjugated to any $\mu^k$, 
as in Figure \ref{fig:rho} (with $n=4$).
\begin{figure}
   \begin{center}
       \epsfig{figure=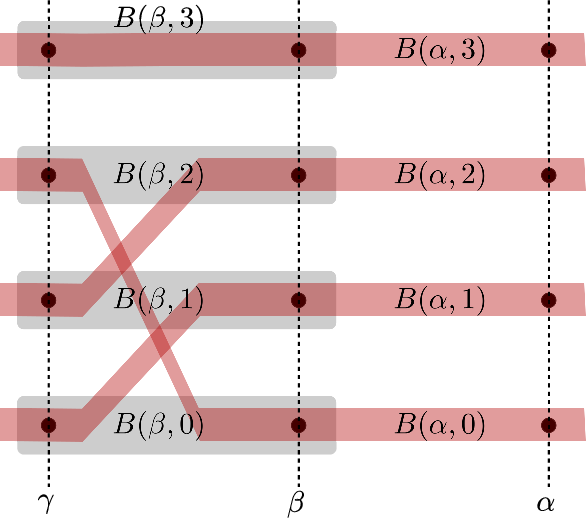, width=.45\textwidth}
       \caption{An example where Lemma \ref{lemma:rho} cannot hold.}
       \label{fig:rho}
   \end{center}
\end{figure}

\begin{proof}
   Let $\Upsilon = \Upsilon(\mathscr{B})$.
   The condition $\mu \bm{\cdot} B(\alpha,i) \cap A = B(\alpha,\mu_\alpha(i))\cap A$
   means in particular (when looking at $\{\beta\}\times n$)
   that $\mu_\beta(\nu_\alpha(\beta)(i)) = \nu_\alpha(\beta)(\mu_\alpha(i)) $, or in other words
   \begin{equation}
     \label{eq:boringone}
     \tag{$*$}
     \mu_\beta \cdot \nu_\alpha(\beta) = \nu_\alpha(\beta)\cdot \mu_\alpha.
   \end{equation}
   When looking at $\{\gamma\}\times n$, this yields 
   \begin{equation}
     \label{eq:boringtwo}
     \tag{$**$}
      \mu_\gamma \cdot \nu_\alpha(\gamma) = \nu_\alpha(\gamma)\cdot \mu_\alpha.
   \end{equation}
   Finally, the condition $\mu \bm{\cdot} B(\beta,i) \cap A= B(\beta,\mu_\alpha(i))\cap A$
   at $\{\gamma\}\times n$ gives
   \begin{equation}
     \label{eq:boringthree} 
     \tag{$***$}
     \mu_\gamma \cdot \nu_\beta(\gamma) = \nu_\beta(\gamma)\cdot \mu_\alpha.
   \end{equation}
   Then, (i)--(ii)--(iii) imply that $\nu_\alpha(\beta) = \nu_\beta(\gamma) = id$ and $\nu_\alpha(\gamma) = \rho$.
   Hence, by (\ref{eq:boringone}), (\ref{eq:boringtwo}) and (\ref{eq:boringthree}),
   we have that $\mu_\alpha = \mu_\beta=\mu_\gamma$ and $\mu_\gamma \rho = \rho\mu_\gamma$.
   Now, by Proposition \ref{prop:lifeiseasy}, 
   $\mu_\gamma = \sigma\mu\sigma^{-1}$ for some $\sigma$, hence 
   $$ 
      \sigma\mu\sigma^{-1} \rho = \rho \sigma\mu\sigma^{-1}
      \quad\Rightarrow\quad
      \mu(\sigma^{-1}\rho\sigma) = (\sigma^{-1}\rho\sigma) \mu.
   $$
   By Lemma \ref{lemma:conjmu}, $\sigma^{-1}\rho\sigma = \mu^k$, which is what we wanted to show.
\end{proof}

It is now rather easy to define a transverse family which is not principal: it suffices to
choose a $\rho$ not conjugated with $\mu$ and to define $B(\alpha,i)$ so that (i), (ii) and (iii)
of Lemma \ref{lemma:rho} hold for sufficiently many $\gamma<\beta<\alpha$.

\begin{example}
   \label{ex:simnotcyc}
   A family $\mathscr{B}$ in $\omega_1\times n$ ($n\ge 3$)
   which is transverse but not principal.
\end{example}
\begin{proof}[Details]
   Fix a $\rho\in\mathfrak{S}_n$ which is not conjugated to $\mu^k$ for any $k$.
   Define $B(\alpha,i)$ by induction as follows. The successor steps are as usual: 
   $B(\alpha+1,i) = B(\alpha,i)\cup\{\langle\alpha,i\rangle\}$.
   We assume (induction hypothesis):
   \begin{equation}
     \begin{array}{l}
     \forall \xi<\alpha,\,\xi\in\Lambda,\, 
     \text{ there are sequences }\gamma_m<\beta_m<\gamma_{m+1}\text{ with limit }\xi\\
     \text{ such that }  \forall i\,
     B(\xi,i)\ni\langle\beta_m,i\rangle,\, B(\beta_m,i)\ni\langle\gamma_m,i\rangle
       \text{ and }B(\xi,i)\ni\langle\gamma_m,\rho(i)\rangle.
     \end{array}
     \label{eq:induction}
   \end{equation}
   If $\alpha = \beta+\omega$ with $\beta\in\Lambda\cup\{0\}$, 
   it is easy to make (\ref{eq:induction}) work for $\xi=\alpha$ by choosing $\gamma_n=\beta + 2n$,
   $\beta_n = \beta+2n+1$ and defining $B(\alpha,i)$ accordingly.
   If $\alpha\in\Lambda_2$, choose an increasing sequence $\beta_m\in\Lambda$ converging to $\alpha$.
   Then, pick some $\gamma_m$ with $\beta_{m-1}<\gamma_m<\beta_m$ such that 
   $B(\beta_m,i)\ni\langle\gamma_m,i\rangle$ (which exists by the induction hypothesis).
   If $\gamma_m= \eta + 1$ for some $\eta$, let $\eta_m = \eta$.
   Otherwise, let $\eta_m$ be $\beta_{m-1}$.
   Let $A(\beta_m,i)$ be equal to $B(\beta_m,i)$ except 
   in the interval $(\eta_m,\gamma_m]$ 
   where it contains $(\eta_m,\gamma_m]\times\{\rho(i)\}$.
   Set now
   $$
      B(\alpha,i) = \{\langle\alpha,i\rangle\}\cup\bigcup_{m\in\omega} A(\beta_m,i)\cap (\beta_{m-1},\beta_m]\times n.
   $$
   Then $B(\alpha,i)$ satisfies (\ref{eq:Balpha}) and (\ref{eq:induction}).
   Hence,
   $K(\mathscr{B})$ is indeed a transverse $n$-cp$\omega_1$.
   \\
   Suppose now that there is a continuous free transitive action by $\Z_n$ on the $K(\mathscr{B})$.
   By Lemma \ref{lemma:actionbeta} and Fodor, there is $\beta$ and a stationary set $S$ of $\alpha$ 
   such that the action by $\mu$ sends $B(\alpha,i)\cap A$ to $B(\alpha,\mu_\alpha(i))\cap A$ when
   $\alpha\in S$, with $A=(\beta,\alpha]\times n$.
   Take $\alpha\in S\cap\Lambda_2$ and $m$ such that $\beta<\gamma_m$ given by (\ref{eq:induction}).
   We are then exactly in the situation of Lemma \ref{lemma:rho},
   with $\beta=\beta_m$ and $\gamma=\gamma_m$.
   Since $\rho$ is not conjugate to $\mu^k$, 
   no such action exists.
\end{proof}

Since a club $C\subset\omega_1$ is homeomorphic to $\omega_1$, 
if $K$ is an $n$-cp$\omega_1$ then $K\cap (C\times n)$ is also an $n$-cp$\omega_1$.
It may happen that a non-principal transverse family becomes principal when restricted to a club
set of fibers.
Using our old friend $\clubsuit_C$, we may obtain an example where this does not happen.

\begin{example}[{$\clubsuit_C$}]
   \label{ex:simnotcycclub}
   A transverse family $\mathscr{B}$ in $\omega_1\times n$ ($n\ge 3$)
   whose restriction to $C\times n$ is not principal, for any club $C\subset\omega_1$.
\end{example}
\begin{proof}[Details]
   Exactly the same construction as the previous example, except that we choose $\gamma_n,\beta_n$
   according to a $\clubsuit_C$ sequence $\mathscr{L}=\langle L_\alpha\,:\,\alpha\in\Lambda\rangle$.
   That is, 
   when $\alpha\in\Lambda_2$ and $\beta_m$ list $L_\alpha$ in increasing order, we
   take $\gamma_m\in L_{\beta_m}$ such that $\beta_{m-1}<\gamma_m<\beta_m$.
   If there is an action of $\Z_n$ on $K(\mathscr{B})\cap (C\times n)$ for some club $C$, 
   then there is a stationary set $S_0$ such that $L_\alpha\subset C$ for $\alpha\in S_0$.
   There is $\beta\in\omega_1$ such that
   for a stationary subset $S_1\subset S_0$, 
   the action by $\mu$ sends $B(\alpha,i)\cap A$ to $B(\alpha,\mu_\alpha(i))\cap A$,
   where $A=((\beta,\alpha]\cap C)\times n$.
   Take now $\alpha\in (S_1\cap C)'$ to conclude as in the previous example.
\end{proof}

\begin{q}
   Is $\clubsuit_C$ necessary in Example \ref{ex:simnotcycclub}~?
\end{q}

%%%%%%%%%%%%%%%%%%%%%%%%%%%%%%%%%%%%%%%%%%%%%%%%%%%%%%%%%%%%%%%%%%%%%%%%%%%%%%%%%%%%%%%%%%%%%%
\subsection{A non-normal principal $\Z_n$-bundle}

We describe here yet another Nyikos construction: that of a principal $\Z_n$-bundle which is non-normal
in {\bf ZFC}. He actually only gave an example when $n=2$ in \cite[Example 8.1]{Nyikoscoherent}, but
it is easy to generalize it to any $n$. The ideas are very similar (and actually easier) to what we did 
in the previous subsection.
We first notice the following.

\begin{lemma}[Nyikos]
    If a transverse $n$-cp$\omega_1$ contains $n$ disjoint copies of $\omega_1$ and is normal,
    it is homeomorphic to $\omega_1\times n$ with the product topology.
\end{lemma}
\begin{proof}
    There is a club $C\subset\omega_1$ such that each of the $n$ copies $A_0,\dots,A_{n-1}$ of $\omega_1$ intersects 
    exactly one member of the fibers above $C$, as easily seen.
    If $U_0,\dots,U_{n-1}$ are open disjoint and respectively contain $A_0,\dots,A_{n-1}$,
    by Fodor there is $\beta<\omega_1$ such that $\cup_{i=1,\dots,n} U_i \supset \pi^{-1}((\beta,\omega_1))$.
    Then each $U_i\cap\pi^{-1}((\beta,\omega_1))$ is clopen and provide an homeomorphism
    between $\pi^{-1}((\beta,\omega_1))$ and $(\beta,\omega_1)\times n$ with the product topology.
    We may then conclude with Lemma \ref{lemma:symbundle}.
\end{proof}

\begin{lemma}
  \label{lemma:coherent}
  There is a family of functions $u_\alpha:[0,\alpha]\to n$ ($\alpha\in\Lambda$, interval in $\omega_1$) 
  such that:
  
  (i) $u_\alpha(\beta) = 0 $ for each $\beta\in\Lambda$, $\beta\le\alpha$;
  
  (ii) for each $\beta<\alpha$, 
       the set $D_{\alpha,\beta}=\{\gamma\le\beta\,:\,u_\alpha(\gamma)\not=u_\beta(\gamma)\}$ is finite;
  
  (iii) for each $\alpha\in\Lambda_2$, each increasing sequence $\beta_m\in\Lambda$ with limit $\alpha$
        and each $i=1,\dots,n-1$ there is a sequence $\gamma_{m,i}<\beta_m$ with $\lim_{m\to\infty}\gamma_{m,i}=\alpha$
        such that $u_\alpha(\gamma_{m,i}) =  u_{\beta_m}(\gamma_{m,i}) + i \text{ (mod }n)$.
\end{lemma}
\begin{proof}
    We build $u_\alpha$ and a ladder system $L_\alpha$ together by induction, such that (i) and (ii) hold, together
    with the following two conditions:
    
    (iii') for each $\alpha\in\Lambda_2$ and each $\beta<\alpha$, $\beta\in\Lambda$, there is $k\in\omega$
           such that $\sigma_\alpha(k)<\beta\le\sigma_\alpha(k+1)$ and for each $i=1,\dots,n-1$,
           there is $\gamma_i\in(\sigma_\alpha(k),\beta) - \Lambda$ such that 
           $u_\alpha(\gamma_i) = u_\beta(\gamma_i) + i \text{ (mod }n)$;
    
    (iv) for each $\alpha\in\Lambda_2$ and each $\eta<\alpha$, $\eta\in\Lambda$,
         the set 
         $(\eta,\eta+\omega) - (\cup_{\beta\in\Lambda\cap(\eta,\alpha)} D_{\alpha,\beta})$
         is infinite.
    \vskip .2cm
    It is immediate that (iii') implies (iii). 
    Let us start the construction by setting $u_0(0)=0$.
    If $\alpha = \beta+\omega$ for some $\beta\in\Lambda\cup\{0\}$, 
    set $u_\alpha(\gamma) =0 $ for each $\gamma\in(\beta,\alpha]$
    and $u_\alpha\upharpoonright[0,\beta] = u_\beta$.
    Let also $L_\alpha$ be $\{0\}\cup\{\beta+m\,:\,m\in\omega\}$.
    If $\alpha\in\Lambda_2$ but $\Lambda_2$ is not cofinal in $[0,\alpha)$,
    then $\alpha = \beta+\omega\cdot\omega$ for some $\beta\in\Lambda_2\cup\{0\}$.
    In that case set $L_\alpha=\{0\}\cup\{\beta + \omega\cdot m\,:\,m\in\omega\}$,
    and define $u_\alpha( \beta + \omega\cdot m + i ) = i$ for each $i=1,\dots,n-1$,
    $u_\alpha( \gamma ) = 0$ for each other $\gamma\in(\beta,\alpha]$.
    By (iv) for $\beta$, there are minimal $\gamma_1,\dots,\gamma_{n-1}\in(0,\omega)$ such that
    $u_\beta(\gamma_i) = u_\eta(\gamma_i)$ for each $\eta<\beta$ in $\Lambda$.
    We define $u_\alpha\upharpoonright[0,\beta] = u_\beta$ except at $\gamma_i$
    where it takes the value $u_\beta(\gamma_i) + i\text{ (mod }n)$.
    Then (i), (ii), (iii') and (iv) hold by construction (and induction).\\
    Let now $\alpha$ be such that $\Lambda_2$ is cofinal in $[0,\alpha)$.
    Choose $L_\alpha\subset\{0\}\cup\Lambda_2$, with $0\in L_\alpha$.
    To alleviate the notation, denote its members by $\alpha_k$ ($k\in\omega$).
    By (iv) again, there are minimal $\gamma_1,\dots,\gamma_{n-1}\in(\alpha_k,\alpha_{k}+\omega)$
    such that for all $\beta\in\Lambda$ with $\alpha_k<\beta<\alpha_{k+1}$, we have
    $u_{\alpha_{k+1}}(\gamma_i) = u_\beta(\gamma_i)$.
    Define $u_\alpha$ such that it is equal to $u_{\alpha_{k+1}}$ on $(\alpha_k,\alpha_{k+1}]$
    except at those $\gamma_i$ where it takes the value $u_{\alpha_{k+1}}(\gamma_i) + i \text{ (mod }n)$.
    Then (i) and (ii), and (iii') hold by induction.
    Since $u_\alpha$ differs from $u_{\alpha_{k+1}}$ only at finitely many points,
    (iv) still holds by induction.
\end{proof}

\begin{example}[{Nyikos \cite[Example 8.1]{Nyikoscoherent}}]
   \label{ex:nonnormal}
   For each $n\ge 2$, there is a non-normal principal $\Z_n$-bundle.
\end{example}
\begin{proof}[Details]
   Take the sequence $u_\alpha$ given by Lemma \ref{lemma:coherent}. In $\omega_1\times n$,
   let $\mathscr{B}$ be the family 
   $B(\alpha,i) = \{\langle \beta, u_\alpha(\beta) +i\text{ (mod }n)\rangle\,:\,\beta\le\alpha\}$
   for $\alpha\in\Lambda$, and $B(\alpha+1,i) = \{\langle\alpha+1,i\rangle\}\cup B(\alpha,i)$.
   By (ii), $\mathscr{B}$ satisifes (\ref{eq:Balpha}) and thus engenders a symmetric $n$-cp$\omega_1$.
   The action of $\Z_n$ by addition modulo $n$ in each fiber is continuous, hence 
   $K(\mathscr{B})$ is a principal $\Z_n$-bundle.
   By (i), the topology on $A_i = \Lambda\times\{i\}$ is the product topology, hence the
   $A_i$ are disjoint copies of $\omega_1$.
   But $A_i$ and $A_j$ cannot be put on disjoint open sets if $i<j$. Indeed, if $U,V$ are 
   open and contain respectively $A_i$ and $A_j$, by Fodor there is some $\eta$ and
   a stationary $S$ such that $U\supset B(\alpha,i)\cap\pi^{-1}((\eta,\alpha])$ and
   $V\supset B(\alpha,j)\cap\pi^{-1}((\eta,\alpha])$ for each $\alpha\in S$.
   Take $\alpha\in S''$ with $\alpha>\eta$. Then $\alpha\in\Lambda_2$ and there is a sequence
   $\beta_m\in\Lambda\cap S$ with limit $\alpha$.
   Then by (iii) there are $m\in\omega$ and $\gamma\in(\eta,\beta_m)$ such that 
   $u_\alpha(\gamma) + i = u_{\beta_m}(\gamma) + j \text{ (mod }n)$.
   By definition, $B(\alpha,i)\cap B(\beta_m,j)\ni\langle \gamma, u_\alpha(\gamma) + i \text{ (mod }n)\rangle$.
   It follows that $U\cap V\not=\varnothing$ and thus $K(\mathscr{B})$ is not normal.
\end{proof}

\begin{cor}
  There is an $n$-sheeted covering map $\pi:K(\mathscr{B})\to\omega_1$
  such that the base space $\omega_1$ is normal while the covering space $K(\mathscr{B})$ is not.
\end{cor}
\begin{proof}
  Example \ref{ex:nonnormal} is non-normal and transverse. By 
  Corollary \ref{cor:covering}, $\pi$ is an $n$-sheeted covering map.
\end{proof}
Of course, once we find a $2$-sheeted covering map with non-normal $K(\mathscr{B})$, adding a copy of 
$\omega_1\times (n-2)$ with the product topology gives an $n$-sheeted such covering map.
The bonus in Example \ref{ex:nonnormal} when $n>2$ is the principal bundle structure.

We notice that sequences of functions $u_\alpha$ satisfying (ii) of Lemma \ref{lemma:coherent}
are called {\em coherent}. One way to obtain a coherent sequence is to start with
a function $u:\omega_1\to n$ and to define $u_\alpha$ by making finitely many changes to 
$u\upharpoonright[0,\alpha]$. If, given the $u_\alpha$, 
we can find such an $u$, the sequence is called {\em uniformizable}.
Another interpretation of the proof that $K(\mathscr{B})$ is non-normal is that the sequence
$u_\alpha$ is {\em not} uniformizable. This similarity can be pushed even further:
$K(\mathscr{B})$ is normal iff $u_\alpha$ is uniformizable.
(This equivalence is the key insight of Nyikos on this subject, see Section 4 in \cite{Nyikoscoherent} for more).
Many 
non-uniformizable coherent
sequences have appeared in the litterature since the 1970s (at least), 
but they usually have range $2$, not an arbitrary finite $n$. 
Since we also wanted our sequence to satisfy (i), we thought it was simpler to 
define it ourselves instead of trying to adapt a published one to our purposes.
Notice that if a non-uniformizable sequence $u_\alpha$ has range $2$, and the
$B(\alpha,i)$ and $A_i$ are defined as above, then $A_i$ cannot
be separated from $A_{i+1}$ but $A_i$ can be separated from $A_{i+2}$. We find that
our version where no pair $A_i,A_j$ can be separated is more elegant.

%%%%%%%%%%%%%%%%%%%%%%%%%%%%%%%%%%%%%%%%%%%%%%%%%%%%%%%%%%%%%%%%%%%%%%%%%%%%%%%%%%%%%%%%%%%%%%
%%%%%%%%%%%%%%%%%%%%%%%%%%%%%%%%%%%%%%%%%%%%%%%%%%%%%%%%%%%%%%%%%%%%%%%%%%%%%%%%%%%%%%%%%%%%%%
%%%%%%%%%%%%%%%%%%%%%%%%%%%%%%%%%%%%%%%%%%%%%%%%%%%%%%%%%%%%%%%%%%%%%%%%%%%%%%%%%%%%%%%%%%%%%%

\section{A variation on $\clubsuit_C$.}
\label{sec:variationclub}

While trying to fill the gaps in Nyikos' draft, we first 
thought that the principle $\clubsuit_C^2$ below was necessary for the proof of Theorem \ref{thm:intro}.
It was not clear to us whether it is consistent, but after asking in Mathoverflow,
G. Goldberg and A. Lietz provided a proof that it holds in the constructible universe
$\mathbf{L}$. We then found a simple proof that the 
more general principle $\clubsuit_C^{<\omega}$
follows from $\diamondsuit$.
This family of principles lack proper applications, but in the spirit of Nyikos' drafts which contain 
many propositions of similar axioms, we decided to include them here.
(In this section, we abandon the interval notation $[0,\alpha)$ and use $\alpha$ instead, as it is more usual
when dealing with these principles.)

\begin{defi}\ \\
   If $\mathscr{L} = \langle L_\alpha\,:\,\alpha\in\omega_1\rangle$ 
   is a ladder system, define by induction $T_\alpha^0(\mathscr{L}) = \{\alpha\}$ and
   $T^{n+1}_\alpha(\mathscr{L}) = T^{n}_\alpha(\mathscr{L})\cup\cup_{\beta\in T^{n}_\alpha(\mathscr{L})}L_\beta$.
   (We assume that $L_\alpha=\varnothing$ if $\alpha$ is successor.)
   Set $T^{\omega}_\alpha(\mathscr{L}) = \cup_{n\in\omega}T^{n}_\alpha(\mathscr{L})$. 
   Then, for $n\le\omega$, we define the following axioms:
   \vskip .2cm\noindent
   \begin{minipage}{.06\textwidth} 
       $\clubsuit^{n}_C$ : 
    \end{minipage} 
    \begin{minipage}{.85\textwidth}              
                     There is a ladder system $\mathscr{L}=\langle L_\alpha\,:\,\alpha\in\Lambda\rangle$
                     such that for each 
                     club $C\subset\omega_1$, there is a stationary set of $\alpha$ such that
                     $T^{n}_\alpha(\mathscr{L})\subset C$.                
    \end{minipage} 
    \vskip .2cm\noindent
    \begin{minipage}{.06\textwidth} 
       $\clubsuit^{<\omega}_C$ : 
    \end{minipage} 
    \begin{minipage}{.85\textwidth}              
                     There is a ladder system $\mathscr{L}=\langle L_\alpha\,:\,\alpha\in\Lambda\rangle$
                     such that for each 
                     club $C\subset\omega_1$, 
                     for each $n\in\omega$
                     there is a stationary set of $\alpha$ such that
                     $T^{n}_\alpha(\mathscr{L})\subset C$.                
    \end{minipage}    
\end{defi}
Recall Jensen's diamond (which holds e.g. in $\mathbf{L}$):
\vskip .2cm\noindent
    \begin{minipage}{.05\textwidth} 
       $\diamondsuit$ : 
    \end{minipage} 
    \begin{minipage}{.85\textwidth}              
                     There is a $\diamondsuit$-sequence $\langle A_\alpha\,:\,\alpha\in\omega_1\rangle$,
                     i.e. for each $E\subset\omega_1$ there is a stationary set of $\alpha$ for which
                     $A_\alpha\cap\alpha = E\cap\alpha$.                
    \end{minipage}                    
\vskip .2cm
Recall that when $E\subset\omega_1$ we denote by $E'$ is its derived set (i.e. its limit points).
Of course, $E'\subset\Lambda$ for any $E\subset\omega_1$ and $S'$ is stationary whenever $S$ is stationary.

\begin{lemma}
   $\clubsuit_C^\omega$ is inconsistant.
\end{lemma}
\begin{proof}
   We show that 
   for any ladder system $\mathscr{L}$,
   $T^\omega_\alpha(\mathscr{L})$ contains successor ordinals for each $\alpha\in\Lambda$.
   This is clear if $\alpha=\omega$. Suppose by induction that it is true for each $\beta<\alpha$
   with $\beta,\alpha\in\Lambda$.
   If $L_\alpha$ contains successor ordinals, we are over. Otherwise, $L_\alpha\subset\Lambda$ and thus
   $T^\omega_\beta(\mathscr{L})$ contains successor ordinals for each $\beta\in L_\alpha$.
   It follows that $T^\omega_\alpha(\mathscr{L}) = L_\alpha\cup \cup_{\beta\in L_\alpha}T^\omega_\beta(\mathscr{L})$
   also contains (lots of) successor ordinals.
   This shows that $\clubsuit_C^\omega$ cannot hold, since e.g. $\Lambda$ is a club without successor ordinals.
\end{proof}

\begin{lemma}
    \label{lemma:diamondclub}
    $\diamondsuit \,\Longrightarrow\, \clubsuit^{<\omega}_C$
\end{lemma}
\begin{proof}
   Let $\langle A_\alpha\,:\,\alpha\in\omega_1\rangle$ be a $\diamondsuit$-sequence.
   Let $\alpha\in\Lambda$. If $A_\alpha\cap\Lambda$ is cofinal in $\alpha$, we 
   take a cofinal increasing sequence $L_\alpha\subset A_\alpha\cap\Lambda$.
   If $A_\alpha\cap\Lambda$ is bounded in $\alpha$, take any increasing cofinal sequence $L_\alpha$.
   We now prove that these form a $\clubsuit_C^{<\omega}$-sequence.
   \\
   Let $C\subset\omega_1$ be club.
   We shall define stationary sets $S_n$ by induction on $n$.
   Set $S_0=C$. 
   Given $S_n$, we let $\wt{S}_{n+1} = \{\alpha\in\omega_1\,:\,A_\alpha\cap\alpha = (S_n\cap C)'\cap\alpha\}$.
   Then $\wt{S}_{n+1}$ is stationary.
   Since $(S_n\cap C)'$ is stationary  
   and in particular unbounded in $\omega_1$, the set 
   $D_{n+1} = \{\alpha\in\omega_1\,:\, (S_n\cap C)'\text{ is cofinal in }\alpha\}$
   is club.
   We then set $S_{n+1} = \wt{S}_{n+1}\cap D_{n+1}$, then $S_{n+1}$ is stationary.
   \begin{claim}
      For each $n\ge 0$, 
      $S_n \subset\{\alpha\in\omega_1\,:\,T^{n}_\alpha(\mathscr{L})\subset C\}$.
   \end{claim}
   \begin{proof}
      By induction. 
      We abbreviate $T^{n}_\alpha(\mathscr{L})$ by $T^{n}_\alpha$ as $\mathscr{L}$ is clear from the context.
      If $\alpha \in S_0=C$, this is clear since $T^0_\alpha=\{\alpha\}$.  
      Assume that the claim holds for each $m<n$.
      If $\alpha\in S_n$, then $(S_{n-1}\cap C)'$ is cofinal in $\alpha$ and 
      $\Lambda\cap A_\alpha\cap\alpha = \Lambda\cap (S_{n-1}\cap C)'\cap\alpha = (S_{n-1}\cap C)'\cap\alpha$.
      It follows that $\Lambda\cap A_\alpha$ is cofinal in $\alpha$, hence
      $L_\alpha\subset A_\alpha\subset (S_{n-1}\cap C)'$. In particular, $L_\alpha\subset S_{n-1}\cap C$.
      By induction, $T^{n-1}_\beta\subset C$ for each $\beta\in L_\alpha$.
      Since $L_\alpha\subset C$ as well, by closedness $\alpha\in C$. 
      The claim follows since
      $T^{n}_\alpha = \{\alpha\}\cup L_\alpha \cup_{\beta\in L_\alpha} T^{n-1}_\beta$. 
   \end{proof}
   \noindent
   This shows that $\{\alpha\in\omega_1\,:\,T^{n}_\alpha(\mathscr{L})\subset C\}$
   contains a stationary set, which finishes the proof.
\end{proof}

Of course, $\clubsuit^{<\omega}_C\Rightarrow\clubsuit^{n}_C\Rightarrow\clubsuit^{m}_C$
for each $m<n<\omega$.
\begin{q} 
   Does $\clubsuit_C^m$ imply $\clubsuit_C^n$ when $m<n$~? Does it imply $\clubsuit_C^{<\infty}$~?
\end{q}

\bibliographystyle{plain}
\bibliography{../biblio}

\end{document}